\documentclass[11pt,bibliography=totoc,reqno]{scrartcl}

\usepackage[a4paper,left=2cm,right=2cm,top=3cm,bottom=3cm]{geometry}

\usepackage[ansinew]{inputenc}
\usepackage[T1]{fontenc}
\usepackage[english]{babel}
\usepackage{marvosym}
\usepackage{stmaryrd}
\usepackage{caption}

\usepackage{units}
\usepackage{ifsym}
\usepackage[intlimits,leqno]{amsmath}
\usepackage{amsthm}
\usepackage{amssymb}
\usepackage{amsfonts}
\usepackage{dsfont}
\usepackage{exscale}
\usepackage{txfonts}
\usepackage{pxfonts}
\usepackage{bbold}
\usepackage{bbm}
\usepackage{wrapfig}
\usepackage{wasysym}
\usepackage{mathrsfs}
\usepackage{setspace}
\usepackage{todonotes}
\usepackage[all]{xy}
\usepackage{float}

\usepackage[numbers]{natbib}

\usepackage{multicol}

\newcommand{\e}{\operatorname{e}}
\newcommand{\rd}{\mathrm{d}}

\newcommand{\im}{\operatorname{im}}
\newcommand{\id}{\operatorname{id}}

\newcommand{\diag}{\operatorname{diag}}

\DeclareMathOperator{\tr}{tr}

\newcommand{\Hom}{\operatorname{Hom}}

\newtheorem{theorem}{Theorem}
\newtheorem{remark}{Remark}
\newtheorem{lemma}[remark]{Lemma}
\newtheorem{proposition}[remark]{Proposition}
\theoremstyle{definition}
\newtheorem{definition}[remark]{Definition}
\newtheorem{example}[remark]{Example}

\usepackage{scrpage2} 
\automark[subsection]{section}

\begin{document}
\pagenumbering{arabic} 


\title{Metric Symplectic Lie Algebras\footnote{funded by the Deutsche Forschungsgemeinschaft (DFG, German Research Foundation) - KA 1065/7-1}}
\author{Mathias Fischer}
\maketitle
\onehalfspacing 


\abstract{Every metric symplectic Lie algebra has the structure of a quadratic extension. We give a standard model and describe the equivalence classes on the level of corresponding quadratic cohomology sets. Finally, we give a scheme to classify the isomorphism classes of metric symplectic Lie algebras and give a complete list of all these Lie algebras in special cases.}

\section{Introduction}
In this paper we concentrate on metric symplectic Lie algebras $(\mathfrak g,\langle\cdot,\cdot\rangle,\omega)$. These are symplectic Lie algebras $(\mathfrak g,\omega)$ which are also metric Lie algebras $(\mathfrak g,\langle\cdot,\cdot\rangle)$. Of course, an isomorphism of metric symplectic Lie algebras is an isomorphism of the corresponding symplectic Lie algebras which is also an isomorphism of the corresponding metric Lie algebras.
We call a metric symplectic Lie algebra decomposable if it is isomorphic to the direct sum of two non-trivial metric symplectic Lie algebras.

Here a symplectic Lie algebra $(\mathfrak g,\omega)$ is a Lie algebra $\mathfrak g$ admitting a closed non-degenerate $2$-form $\omega$ on $\mathfrak g$, which we call symplectic form.
Two symplectic Lie algebras $(\mathfrak g_i,\omega_i)$, $i=1,2$ are isomorphic if there is an isomorphism $\varphi:\mathfrak g_1\rightarrow\mathfrak g_2$ of Lie algebras which preserves the symplectic forms in the sense $\varphi^*\omega_2=\omega_1$.
Symplectic Lie algebras are in one-to-one correspondence with simply connected Lie groups with leftinvariant symplectic forms.
Symplectic Lie algebras are also called quasi-Frobenius Lie algebras, since Frobenius Lie Algebras, i. e. Lie algebras admitting an non-degenerate exact $2$-form, are examples of symplectic Lie algebras.
Moreover, symplectic Lie algebras are examples of Vinberg algebras, since they naturally carry an affine structure.

A metric Lie algebra $(\mathfrak g,\langle\cdot,\cdot\rangle)$ is a Lie algebra $\mathfrak g$ with an non-degenerate symmetric bilinear form $\langle\cdot,\cdot\rangle$ on $\mathfrak g$, which is ad-invariant, i. e. 
\[\langle[X,Y],Z\rangle=\langle X,[Y,Z]\rangle\]
for all $X,Y,Z\in\mathfrak g$. Allthough, $\langle\cdot,\cdot\rangle$ is not necessary positive definite, we will call it inner product.
In the literature these Lie algebras are also called quadratic Lie algebras.
Two metric Lie algebras are isomorphic if there is an isomorphism of the corresponding Lie algebras, which is also an isometry of the inner products.

Metric Lie algebras are in one-to-one correspondence with simply connected Lie groups with biinvariant metrics.
They are also used for describing symmetric spaces,
since the Lie algebra of the transvection group of a symmetric space admits such a structure \cite[Proposition 1.6]{CP80}.

There are several classifications of metric or symplectic Lie algebras in low dimension (\cite{Sa01}, \cite{GJK01}, \cite{Ov06}, \cite{CS09}).
Especially, the nilpotent case (\cite{Bu06}, \cite{Ka07}, \cite{GKM04}, see also \cite{GB87}) is importent for our problem, since every metric symplectic Lie algebra is nilpotent.

If the aim is to give informations for arbitrary dimension, then usually a reduction scheme is used. In this context we mention 
double extensions (\cite{MR83},\cite{MR85}), $T^*$-extensions \cite{Bo97},  symplectic double extensions (\cite{MR91}, \cite{DM96}, \cite{DM962}) and oxidation \cite{BC13}.
The main idea is, that for every isotropic ideal $\mathfrak j$ of a metric or symplectic Lie algebra $\mathfrak g$ the Lie algebra $\overline{\mathfrak g}=\mathfrak j^\bot/\mathfrak j$ inherits an inner product or symplectic form respectively from $\mathfrak g$.
Conversely, they give a construction scheme taking a metric or symplectic Lie algebra $\overline{\mathfrak{g}}$ and some additional structure and construct a higher dimensional metric respectively symplectic Lie algebra, which can be reduced to $\overline{\mathfrak{g}}$ again.
Since the choice of the isotropic ideal is in general not canonical, it is hard to to give a general statement on the isomorphy of this Lie algebras with the help of these schemes. For instance, it is possible that extensions of two non-isomorphic low dimensional metric or symplectic Lie algebras are isomorphic.
Moreover, the presentation of a metric or symplectic Lie algebra as such an extension is not unique, since it depends on the choosen ideal.

The aim of this paper is to take this choices canonically for metric symplectic Lie algebras such that there is a certain standard model and the possibility to analyse the isomorphy systematically with this standard model.

Until now, there is just a few literature about metric symplectic Lie algebras (for example \cite{BBM07}).
These Lie algebras are in one-to-one correspondence to nilpotent metric Lie algebras with bijective skewsymmetric derivations.
On every Lie group with Lie algebra $\mathfrak g$ there is a flat and torsion-free connection $\nabla^\omega$ induced by $\omega$ and an left-invariant pseudo-Riemannian metric whose Levi-Civita connection equals $\nabla^\omega$ \cite{MR91}.

A classification scheme for metric Lie algebras from Kath and Olbrich denoted as quadratic extension is very useful for our aim. 
This scheme was first introduced in \cite{KO06}, where metric Lie algebras were discussed comprehensively.
Using the main idea of \cite{KO06} they introduced a classification scheme in \cite{KO08}, which can be used for metric Lie algebras and metric Lie algebras with additional structure.
Especially, metric Lie algebras with semisimple skewsymmetric derivations were treated in that paper in the notion of $(\mathds R,\{e\})$-equivariant metric Lie algebras.
Thus, this is useful for our problem.
Our aim is to expand the classification scheme in \cite{KO08} for metric Lie algebras with semisimple skewsymmetric derivations in such a way that it classifies all metric symplectic Lie algebras.
\\\quad\\
In the following, we describe the main idea of the present paper.
For every metric Lie algebra $\mathfrak g$ with bijective skewsymmetric derivation $D$ there is an $D$-invariant isotropic ideal $\mathfrak{i}$ of $\mathfrak g$ in a canonical way such that $\mathfrak i^\bot/\mathfrak i$ is abelian. This ideal is given by 
\begin{align}
\mathfrak i=\mathfrak i(\mathfrak g):=\sum_{k=0}^{n-1}\mathfrak g^{k+1}\cap{\mathfrak g^{k+1}}^\bot
\end{align}
(compare \cite{Ka07}).
Here $\mathfrak g^1:=\mathfrak g,\dots,\mathfrak g^k:=[\mathfrak g,\mathfrak g^{k-1}]$ denotes the decreasing central series of the nilpotent Lie algebra $\mathfrak g$ and $n$ the smallest positive integer such that $\mathfrak g^n=\{0\}$.
For the definition of this ideal in the case of non-nilpotent metric Lie algebras see \cite{KO06} or \cite{KO08}.
We set $\mathfrak a=\mathfrak i^\bot/\mathfrak i$ and $\mathfrak l=\mathfrak g/\mathfrak i^\bot$ and obtain that $\mathfrak a$ is an abelian metric symplectic Lie algebra and inherits the structure of a trivial $\mathfrak l$-module. 
Since $\mathfrak i$ is isomorphic to $\mathfrak l^*$, we can write $\mathfrak g$ as two abelian extensions of Lie algebras with bijective derivations
\begin{align}\label{eq:2aE}
0\rightarrow\mathfrak l^*\rightarrow\mathfrak g\rightarrow\mathfrak h\rightarrow 0
,\quad 0\rightarrow\mathfrak a\rightarrow\mathfrak h\rightarrow\mathfrak l\rightarrow 0
\end{align}
in a canonical way.
Here $\mathfrak h=\mathfrak g/\mathfrak i$.
Conversely, two abelian extensions given as in (\ref{eq:2aE}) define a metric symplectic Lie algebra for every Lie algebra $\mathfrak l$ with bijective derivation and $\mathfrak l$-module $\mathfrak a$, unless the abelian extensions satisfy certain compatibility conditions.
Then the image of $\mathfrak l^*$ in $\mathfrak g$ is usually not equal to the canonical isotropic ideal.

We will introduce this construction scheme under the notion of quadratic extensions.
In general, not every quadratic extension of $\mathfrak l$ by $\mathfrak a$ is given by the canonical isotropic ideal $\mathfrak i(\mathfrak g)$.
Thus we also introduce balanced quadratic extensions. We call a quadratic extension balanced, if the image of $\mathfrak l^*$ in $\mathfrak g$ equals the canonical isotropic ideal.
Every metric symplectic Lie algebra has the structure of a balanced quadratic extension in a canonical way.
There is a natural equivalence relation on the set of quadratic extensions of $\mathfrak l$ by $\mathfrak a$.
Moreover, we define a non-linear cohomology $H^p_{Q+}(\mathfrak{l},D_\mathfrak{l},\mathfrak{a})$ of the Lie algebra $\mathfrak l$ with coefficients in $\mathfrak a$, which includes the cohomology $H^p_{Q}(\mathfrak{l},\phi_\mathfrak{l},\mathfrak{a})$ considered in \cite{KO08}. Then we prove that the equivalence classes of quadratic extensions of $(\mathfrak l,D_\mathfrak l)$ by $(\mathfrak a,D_\mathfrak a)$ are in bijection to the second cohomology set $H^2_{Q+}(\mathfrak{l},D_\mathfrak{l},\mathfrak{a})$.
For the proof we cite necessary results for metric Lie algebras with semisimple skewsymmetric derivations and fit this to the case of not necessary semisimple derivations.
Moreover, we give the standard model $(\mathfrak{d}_{\alpha,\gamma},D_{\delta,\epsilon})$ of a metric Lie algebra with skewsymmetric derivation, which defines a quadratic extension of $\mathfrak{l}$ by $\mathfrak{a}$ for every $[\alpha,\gamma,\delta,\epsilon]\in H^2_{Q+}(\mathfrak{l},D_\mathfrak{l},\mathfrak{a})$.

The equivalence classes of balanced quadratic extensions of $(\mathfrak l,D_\mathfrak l)$ by $(\mathfrak a,D_\mathfrak a)$ are described by $H^2_{Q+}(\mathfrak{l},D_\mathfrak{l},\mathfrak{a})_b\subset H^2_{Q+}(\mathfrak{l},D_\mathfrak{l},\mathfrak{a})$ on the level of the cohomology classes.
We also call these cocycles balanced.
Since balanced is a property of a quadratic extension which does only depend on the structure of the corresponding metric Lie algebra, we can use \cite{KO06} for describing balanced cocycles, or more precisely \cite{Ka07} for nilpotent, metric Lie algebras.
We also have the notion of indecomposable balanced cohomology classes, contained in $H^2_{Q+}(\mathfrak l,D_\mathfrak l,\mathfrak a)_0\subset H^2_{Q+}(\mathfrak l,D_\mathfrak l,\mathfrak a)_0$.
The automorphism group $G(\mathfrak{l},D_\mathfrak{l},\mathfrak{a})$ of the pair $((\mathfrak l,D_\mathfrak l),\mathfrak a)$ acts on $H^2_{Q+}(\mathfrak l,D_\mathfrak l,\mathfrak a)_0$.
Altogether, the isomorphism classes of metric, symplectic Lie algebras are in one-to-one correspondence to 
\begin{align*}
\coprod_{(\mathfrak l,D_\mathfrak l)\in \mathcal L}
\coprod_{\mathfrak a\in \mathcal A_{\mathfrak l,D_\mathfrak l}}
H^2_{Q+}(\mathfrak l,D_\mathfrak l,\mathfrak a)_0/G(\mathfrak l,D_\mathfrak l,\mathfrak a),
\end{align*}
where $\mathcal L$ is a system of representatives of the isomorphism classes of nilpotent Lie algebras with bijective derivations and $\mathcal A_{\mathfrak l,D_\mathfrak l}$ a system of representatives of the isomorphism classes of abelian metric Lie algebras with bijective skewsymmetric derivations (considered as trivial $\mathfrak l$-modules).
Using this classification scheme we obtain the following results:
\begin{itemize}
 \item There is only one non-abelian metric Lie algebra (up to isomorphism) of dimension less than eight, which admits symplecic forms.
 \item There are no metric symplectic Lie algebras whose inner product has an index of one or two, except for abelian ones.
 \item We calculate all isomorphism classes of metric symplectic Lie algebras of dimension less than ten und
 \item we calculate every non-abelian metric symplectic Lie algebra with an index of three up to isomorphism.
\end{itemize}
\quad\\
The paper is organized as follows.
We introduce the notion of (balanced) quadratic extensions in section \ref{MSLA:QuadrErw} and show that every metric symplectic Lie algebra has the structure of a (balanced) quadratic extension in a canonical way.
Furthermore, we give an equivalence relation on the set of quadratic extensions.
Then, we define the quadratic cohomology $H^p_{Q+}(\mathfrak{l},D_\mathfrak{l},\mathfrak{a})$ in section \ref{MSLA:Kohom}.
We also define the isomorphism of pairs and give there action on the quadratic cohomology.
In section \ref{MSLA:Std} we define the standard model $(\mathfrak{d}_{\alpha,\gamma},D_{\delta,\epsilon})$ and show necessary and sufficient conditions on the level of the corresponding cocycles, when $(\mathfrak{d}_{\alpha,\gamma},D_{\delta,\epsilon})$ has the structure of a (balanced) quadratic extension.
We show that every quadratic extension is equivalent to a suitable standard model (section \ref{MSLA:ÄquivStd}) and describe the equivalence of standard models (section \ref{MSLA:Äquivkl}).
In the end of section \ref{MSLA:Äquivkl} we obtain a bijection between $H^2_{Q+}(\mathfrak{l},D_\mathfrak{l},\mathfrak{a})$ and the equivalence classes of quadratic extensions of $\mathfrak{l}$ by $\mathfrak{a}$.
In section \ref{MSLA:Isomkl} we describe the isomorphy of standard models on the level of the corresponding cohomology classes and obtain, finally, the classification scheme for metric symplectic Lie algebras.
As an application, we calculate in section \ref{MSLA:Anw} all non-abelian metric symplectic Lie algebras (up to isomorphisms) whose index of the inner product is less than four.
Moreover, we give a system of representatives of all non-abelian metric symplectic Lie algebras of dimension less than ten.

\section{Metric symplectic Lie algebras}\label{MSLA}

\begin{definition}
A metric, symplectic Lie algebra $(\mathfrak{g},\langle\cdot,\cdot\rangle,\omega)$ is a Lie algebra $\mathfrak g$ with a non-degenerate (not necessary positive definite) symmetric bilinear form $\langle\cdot,\cdot\rangle:\mathfrak{g}\times\mathfrak{g}\rightarrow\mathds{R}$ 
and a non-degenerate skewsymmetric bilinear form $\omega:\mathfrak{g}\times\mathfrak{g}\rightarrow\mathds{R}$ satisfying 
\[\langle X,[Y,Z]\rangle=\langle[X,Y],Z \rangle\quad\text{and}\]
\[\rd\omega(X,Y,Z)=-\omega([X,Y],Z)+\omega([X,Z],Y)-\omega([Y,Z],X)=0\]
for all $X,Y,Z\in\mathfrak g$.
Usually, we will call $\langle\cdot,\cdot\rangle$ inner product and $\omega$ symplectic form.
\end{definition}

\begin{definition}
An isomorphism $\varphi$ between metric symplectic Lie algebras $(\mathfrak{g}_1,\langle\cdot,\cdot\rangle_1,\omega_1)$ and $(\mathfrak{g}_2,\langle\cdot,\cdot\rangle_2,\omega_2)$ is an isomorphism $\varphi:\mathfrak{g}_1\rightarrow\mathfrak{g}_2$ of the corresponding Lie algebras, which is an isometry, i. e. $\varphi^*\langle\cdot,\cdot\rangle_2=\langle\cdot,\cdot \rangle_1$, and satisfies $\varphi^*\omega_2=\omega_1$.
\end{definition}

\section{Quadratic extensions}\label{MSLA:QuadrErw}

In this section we introduce the necessary notion of quadratic extensions for the new classification scheme for metric symplectic Lie algebras. Moreover, we show that every metric symplectic Lie algebra has the structure of a quadratic extension in a canonical way.
\\\quad\\
Let $(\mathfrak g,\langle\cdot,\cdot\rangle,\omega)$ be a metric symplectic Lie algebra. Then
\begin{align}\label{eq:omegaD}
\omega=\langle\cdot,D\cdot\rangle
\end{align}
defines a bijective skewsymmetric map $D$ on $\mathfrak g$, since $\langle\cdot,\cdot\rangle$ is non-degenerate. Here $D$ is called skewsymmetric, if $\langle DX,Y\rangle=-\langle X,DY\rangle$ for all $X,Y\in\mathfrak g$. Moreover, $D$ is a derivation on $\mathfrak g$, since $\omega$ is closed. Now, the existence of a bijective derivation on $\mathfrak g$ implies, that $\mathfrak g$ is nilpotent \cite[Theorem 3]{Ja55}.
Conversely, every bijective skewsymmetric derivation of a nilpotent metric Lie-Algebra $\mathfrak g$ defines a symplectic form $\mathfrak g$ by equation (\ref{eq:omegaD}).
Thus we obtain the following Lemma.
\begin{lemma}\label{lm:msla1}
Metric symplectic Lie algebras are in one-to-one correspondence with nilpotent metric Lie-Algebras with skewsymmetric bijective derivations.
\end{lemma}

From now on, we conzentrate on metric Lie algebras $(\mathfrak{g},\langle\cdot,\cdot\rangle)$ with skewsymmetric derivation $D$ and write abbreviatory $(\mathfrak{g},\langle\cdot,\cdot\rangle,D)$, $(\mathfrak g,D)$ or simply $\mathfrak g$ unless it is clear from the context that this is a metric Lie algebra with a skewsymmetic derivation.

\begin{definition}
An isomorphism (homomorphism) $\varphi$ between Lie algebras with derivations $(\mathfrak{g}_1,D_1)$ and $(\mathfrak{g}_2,D_2)$ is an isomorphism (homomorphism) $\varphi:\mathfrak{g}_1\rightarrow\mathfrak{g}_2$ of the corresponding Lie algebras, which satisfies $D_2\varphi=\varphi D_1$.
An isomorphism between metric Lie algebras with derivations is an isomorphism of Lie algebras with derivations, which is in addition an isometry.
\end{definition}

\begin{lemma}\label{lm:33}
Let $(\mathfrak g,D)$ be a metric Lie algebra with derivation. Then the semisimple part $D_s$ of the Jordan decomposition of the skewsymmetric derivation $D$ is also a skewsymmetric derivation of $\mathfrak g$.
\end{lemma}

\begin{proof}
Consider the complexification of $D$ and $\langle\cdot,\cdot\rangle$, if $D$ has nonreal eigenvalues.
Let $v_1$ and $v_2$ denote two generalized eigenvectors of $D$ corresponding to the eigenvalues $\lambda_1$ and $\lambda_2$.
For a sufficiently large $k\in\mathds R$ we have that 
\begin{align*}
\big(D-(\lambda_1+\lambda_2)\id\big)^k[v_1,v_2]=\sum_{i=0}^k {k \choose i}\big[(D-\lambda_1\id)^{k-i}v_1,(D-\lambda_2\id)^iv_2\big]
\end{align*}
vanishs. Thus $[v_1,v_2]$ is an vector of the generalized eigenspace for $\lambda_1+\lambda_2$ and we obtain $D_s[v_1,v_2]=(\lambda_1+\lambda_2)[v_1,v_2]$.
Thus, $D_s$ is a derivation, since
\[D_s[v_1,v_2]=(\lambda_1+\lambda_2)[v_1,v_2]=[\lambda_1v_1,v_2]+[v_1,\lambda_2v_2]=[D_sv_1,v_2]+[v_1,D_sv_2].\]
Since $D$ is skewsymmetric, the generalized eigenspaces for $\lambda_1$ and $\lambda_2$, $\lambda_1\neq-\lambda_2$ are orthogonal to each other.
This subspaces are invariant under $D_s$. Thus $\langle D_sv_1,v_2\rangle+\langle v_1,D_sv_2\rangle=0$. For $\lambda_1=-\lambda_2$ we obtain
\[\langle D_sv_1,v_2\rangle+\langle v_1,D_sv_2\rangle=\lambda_1\langle v_1,v_2\rangle-\lambda_1\langle v_1,v_2\rangle=0.\] Hence $D_s$ is skewsymmetric.
\end{proof}

\begin{definition}
Let $\mathfrak{l}$ be a Lie algebra.
The triple $(\rho,\mathfrak a,\langle\cdot,\cdot\rangle_\mathfrak a)$ is called orthogonal $\mathfrak l$-module, if $(\mathfrak a,\langle\cdot,\cdot\rangle)$ is an abelian metric Lie algebra and $\rho:\mathfrak l\rightarrow \Hom(\mathfrak a)$ a representation of the Lie algebra $\mathfrak l$ on $\mathfrak a$, which is  skewsymmetric, i. e. $\langle\rho(L)A_1,A_2\rangle=-\langle A_1,\rho(L)A_2\rangle$ for all $L\in\mathfrak l$ and $A_1,A_2\in\mathfrak a$.

Now, let $D_\mathfrak l$ be a derivation on $\mathfrak{l}$.
We call $(\rho,\mathfrak{a},\langle\cdot,\cdot\rangle_\mathfrak a,D_\mathfrak a)$ orthogonal $(\mathfrak l,D_\mathfrak l)$-module, if
$D_\mathfrak a$ is a skewsymmetric map on $\mathfrak a$ and $(\rho,\mathfrak a,\langle\cdot,\cdot\rangle_\mathfrak a)$ an orthogonal $\mathfrak l$-module satisfying $\rho(D_\mathfrak lL)=[D_\mathfrak a,\rho(L)]$ for all $L\in\mathfrak{l}$.
We will also write $(\mathfrak a,D_\mathfrak a)$ abbreviately for an orthogonal $(\mathfrak l,D_\mathfrak l)$-module.

Moreover, if $\rho$ is the trivial representation of $\mathfrak l$ on $\mathfrak a$, then we will call $(\mathfrak{a},\langle\cdot,\cdot\rangle_\mathfrak a,D_\mathfrak a)$ trivial $(\mathfrak l,D_\mathfrak l)$-module.
\end{definition}

\begin{lemma}
Let $\mathfrak{l}$ be a Lie algebra, $D_\mathfrak{l}$ a derivation on $\mathfrak{l}$ and $(\rho,\mathfrak{a},\langle\cdot,\cdot\rangle_\mathfrak a,D_\mathfrak{a})$ an orthogonal $(\mathfrak{l},D_\mathfrak{l})$-module.
Then $(\rho,\mathfrak{a},\langle\cdot,\cdot\rangle_\mathfrak a,{D_\mathfrak{a}}_s)$ is an orthogonal $(\mathfrak{l},{D_\mathfrak{l}}_s)$-module,
where ${D_\mathfrak{l}}_s$ and ${D_\mathfrak{a}}_s$ denote the semisimple parts of the Jordan decomposition of $D_\mathfrak{l}$ and $D_\mathfrak a$.
\end{lemma}

\begin{proof}
Because of Lemma \ref{lm:33} it remains to show that 
\[\rho({D_\mathfrak l}_sL)=[{D_\mathfrak a}_s,\rho(L)]\]
holds for all $L\in\mathfrak{l}$.
Consider the complexification of $D_\mathfrak{a}$, $D_\mathfrak{l}$ and $\rho$ if $D_\mathfrak{a}$ or $D_\mathfrak{l}$ have nonreal eigenvectors.
Let $L\in\mathfrak{l}$ be a vector in the generalized eigenspace of $D_\mathfrak{l}$ for $\lambda$ and $A\in\mathfrak{a}$ in the generalized eigenspace of $D_\mathfrak{a}$ for $\mu$.
Similary to the proof of Lemma \ref{lm:33} we obtain
\[ (D_\mathfrak{a}-(\lambda+\mu)\operatorname{id})^k\rho(L)A=\sum_{i=0}^k {k \choose i} \rho((D_\mathfrak l-\lambda\operatorname{id})^{k-i})(D_\mathfrak a-\mu\operatorname{id})^iA. \]
Thus $(D_\mathfrak{a}-(\lambda+\mu)\operatorname{id})^k\rho(L)A$ vanishes for sufficiently large $k$.
Hence $\rho(L)A$ is a vector in the generalized eigenspace for $\lambda+\mu$.
Finally ${D_\mathfrak{a}}_s\rho(L)A=(\lambda+\mu)\rho(L)A$ and
\[{D_\mathfrak a}_s\rho(L)A-\rho(L){D_\mathfrak a}_sA=\lambda\rho(L)A=\rho({D_\mathfrak l}_sL)A.\]
\end{proof}

\begin{definition}
Let $\mathfrak l$ be a Lie algebra, $D_\mathfrak l$ a derivation of $\mathfrak l$ and $(\rho,\mathfrak{a},\langle\cdot,\cdot\rangle,D_\mathfrak a)$ an orthogonal $(\mathfrak l,D_\mathfrak l)$-module.
A quadratic extension of $(\mathfrak l,D_\mathfrak l)$ by $(\mathfrak a,D_\mathfrak a)$ is a quadrupel $(\mathfrak{g},\mathfrak{i},i,p)$, where
\begin{itemize}
\item $\mathfrak{g}=(\mathfrak{g},\langle\cdot,\cdot\rangle,D)$ is a metric Lie algebra with skewsymmetric derivation,
\item $\mathfrak{i}\subset\mathfrak{g}$ is an isotropic $D$-invariant ideal of $\mathfrak{g}$,
\item $i:(\mathfrak{a},D_\mathfrak a)\rightarrow(\mathfrak{g}/\mathfrak{i},\overline D)$ and $p:(\mathfrak{g}/\mathfrak{i},\overline D)\rightarrow (\mathfrak{l},D_\mathfrak l)$ are homomorphisms of Lie algebras with derivations such that
\[0\rightarrow\mathfrak a\xrightarrow{i}\mathfrak g/\mathfrak i\xrightarrow{p}\mathfrak l\rightarrow 0\]
is a short exact sequence of Lie algebras.
Here $\overline D$ denotes the skewsymmetric derivation of $\mathfrak{g}/\mathfrak{i}$ induced by $D$.
Moreover,
\begin{align}\label{eq:DefRho}
i(\rho(L)A)=[\tilde L,i(A)]\in i(\mathfrak a)
\end{align}
holds for all $\tilde L\in \mathfrak{g}/\mathfrak{i}$ satisfying $p(\tilde L)=L$.
Furthermore, $\im(i)=\mathfrak{i}^{\perp}/\mathfrak{i}$ and $i$ is an isometry onto $\mathfrak{i}^{\perp}/\mathfrak{i}$.
\end{itemize}
\end{definition}

\begin{lemma}
Let $(\mathfrak g,D;\mathfrak i,i,p)$ be a quadratic extension of $(\mathfrak l,D_\mathfrak l)$ by $(\mathfrak a,D_\mathfrak a)$.
Then $(\mathfrak g,D_s;\mathfrak i,i,p)$ is also a quadratic extension of $(\mathfrak l,{D_\mathfrak l}_s)$ by $(\mathfrak a,{D_\mathfrak a}_s)$.
Here $D_s$, ${D_\mathfrak l}_s$ and ${D_\mathfrak a}_s$ denote the semisimple parts of the Jordan decomposition of $D$, $D_\mathfrak{l}$ and $D_\mathfrak a$.
\end{lemma}

\begin{proof}
It is well known that every $D$-invariant subspace is also invariant under the semisimple part $D_s$ of the Jordan decomposition of $D$.
So it is not hard to see that $i:(\mathfrak{a},D_\mathfrak a)\rightarrow(\mathfrak{g}/\mathfrak{i},\overline D)$ and $p:(\mathfrak{g}/\mathfrak{i},\overline D)\rightarrow (\mathfrak{l},D_\mathfrak l)$ are homomorphisms of Lie algebras with corresponding semisimple derivations.
\end{proof}

Let $\mathfrak i$ be an isotropic $D$-invariant ideal of a metric Lie algebra $\mathfrak g$ with bijective skewsymmetric derivation $D$ such that $\mathfrak i^\bot/\mathfrak i$ is abelian. Then the short exact sequence 
\begin{align}\label{eq:kanquadrErw}
0\rightarrow\mathfrak i^\bot/\mathfrak i\xrightarrow{i}\mathfrak g/\mathfrak i\xrightarrow{p}\mathfrak g/\mathfrak i^\bot\rightarrow 0
\end{align}
defines a quadratic extension of $\mathfrak g/\mathfrak i^\bot$ by $\mathfrak i^\bot/\mathfrak i$ with corresponding induced derivations.

\begin{remark}
This was already proved for metric Lie algebras without additional structure in \cite[Page 94]{KO06}. In \cite{KO08}, this statement was generalized for metric Lie algebras with additional structure, the so called $(\mathfrak{h},K)$-equivariant metric Lie algebras.
This also includes metric Lie algebras with semisimple skewsymmetric derivations in the notation of $(\mathds{R},\{e\})$-equivariant metric Lie algebras.
But it is not necessary that the derivations are semisimple to define a quadratic extension using the short exact sequence (\ref{eq:kanquadrErw}).
\end{remark}

We already know that nilpotent metric Lie algebras with bijective skewsymmetric derivations are the main structure to determine the isomorphism classes of metric symplectic Lie algebras.
So, let $\mathfrak{g}$ denote a nilpotent metric Lie algebra and $D$ a skewsymmetric, bijective derivation of $\mathfrak{g}$ for the rest of this section.

We are interested in quadratic extensions whose ideal is the canonical isotropic ideal
\begin{align}
\mathfrak i(\mathfrak g):=\sum_{k=0}^{n-1}\mathfrak g^{k+1}\cap{\mathfrak g^{k+1}}^\bot.
\end{align}
Here $\mathfrak g^1:=\mathfrak g,\dots,\mathfrak g^k:=[\mathfrak g,\mathfrak g^{k-1}]$ denotes the lower central series of $\mathfrak g$ and $n$ the smallest positive integer such that $\mathfrak g^n=\{0\}$.
This is the definition of the canonical isotropic ideal for nilpotent $\mathfrak g$ as it was already used in \cite{Ka07}.
The definition of $\mathfrak i(\mathfrak g)$ for a not necessary nilpotent $\mathfrak g$ is given in \cite[Definition 3.3]{KO06}.

This ideal is isotropic for every nilpotent metric Lie algebra $\mathfrak{g}$, it is $D$-invariant for every derivation $D$ of $\mathfrak{g}$ and moreover $\mathfrak{i}^\bot/\mathfrak{i}$ is abelian \cite[Lemma 3.4 (d)]{KO06} \cite[Proposition 2.7]{KO08}.
Thus we will call this ideal $\mathfrak i(\mathfrak g)$ the canonical isotropic ideal.
Furthermore, it is not hard to prove that every isometry $F:\mathfrak g_1\rightarrow \mathfrak g_2$ of nilpotent metric Lie algebras satisfies $F(\mathfrak i(\mathfrak g_1))=\mathfrak i(\mathfrak g_2)$.

\begin{lemma}(\cite{Ka07})
Let $\mathfrak{g}$ be a nilpotent metric Lie algebra and $\mathfrak i(\mathfrak g)$ the canonical isotropic ideal.
If $(\mathfrak{g},\mathfrak{i}(\mathfrak g),i,p)$ is a quadratic extension of $(\mathfrak{l},D_\mathfrak{l})$ by $(\mathfrak{a},D_\mathfrak{a})$,
then $(\mathfrak{a},D_\mathfrak{a})$ is a trivial $(\mathfrak{l},D_\mathfrak{l})$-module, i. e. $\rho$ is trivial.
\end{lemma}

\begin{definition}
Let $(\mathfrak g,\langle\cdot,\cdot\rangle,D)$ be a nilpotent metric Lie algebra with bijective skewsymmetric derivation $D$.
Moreover, let $\mathfrak{l}$ be a nilpotent Lie algebra, $D_\mathfrak l$ a bijective derivation of $\mathfrak l$ and $(\mathfrak a,D_\mathfrak a)$ a trivial $(\mathfrak l,D_\mathfrak l)$-module, where $D_\mathfrak{a}$ is bijective.
A quadratic extension $(\mathfrak g,\mathfrak i,i,p)$ of $(\mathfrak{l},D_\mathfrak l)$ by $(\mathfrak a,D_\mathfrak a)$ is called balanced,
if $\mathfrak i$ is the canonical isotropic ideal $\mathfrak i=\mathfrak i(\mathfrak g)$.
\end{definition}

\begin{theorem}\label{thm:1}
Every nilpotent metric Lie algebra $\mathfrak g$ with skewsymmetric bijective derivation $D$ has the structure of a balanced quadratic extension.
I. e., there is a nilpotent Lie algebra $\mathfrak l$ with bijektive derivation $D_\mathfrak l$, an abelian metric Lie algebra $\mathfrak a$ with bijective skewsymmetric derivation $D_\mathfrak a$ considered as a trivial $(\mathfrak{l},D_\mathfrak{l})$-module and homomorphisms $i$ and $p$ of Lie algebras such that $(\mathfrak g,\mathfrak i(\mathfrak g),i,p)$ is a balanced quadratic extension of $(\mathfrak l,D_\mathfrak l)$ by $(\mathfrak a,D_\mathfrak a)$.
\end{theorem}

\begin{proof}
We choose $\mathfrak i=\mathfrak{i}(\mathfrak{g})$. This ideal is isotropic, $D$-invariant and $\mathfrak{i}(\mathfrak{g})^\bot/\mathfrak{i}(\mathfrak{g})$ is abelian.
Thus the sequence (\ref{eq:kanquadrErw}) defines a quadratic extension with corresponding derivations, which is balanced by definition.
\end{proof}

It is also possible that $(\mathfrak g,\mathfrak i,i,p)$ is a balanced quadratic extension of the trivial Lie algebra $\{0\}$ by $\mathfrak{a}$.
This means that $\mathfrak i(\mathfrak g)=\{0\}$, which is equivalent to $\mathfrak{g}$ is abelian.
Thus a non-abelian $\mathfrak g$ has the structure of a non-trivial balanced quadratic extension.

There is a natural equivalence relation on the set of quadratic extensions.
\begin{definition}
Two quadratic extensions $(\mathfrak g_j,D_j;\mathfrak i_j,i_j,p_j)$, $j=1,2$ of $(\mathfrak l,D_\mathfrak l)$ by $(\mathfrak a,D_\mathfrak a)$ are equivalent, if there is an isomorphism $F:(\mathfrak g_1,D_1)\rightarrow (\mathfrak g_2,D_2)$ of the metric Lie algebras with derivations such that $F(\mathfrak i_1)=\mathfrak i_2$ and the induced isomorphism $\bar F:\mathfrak g_1/\mathfrak i_1\rightarrow \mathfrak g_2/\mathfrak i_2$ satisfies
\[\bar F \circ i_1=i_2\quad\text{ and }\quad p_2\circ\bar F=p_1.\]
\end{definition}

We will determine this equivalence relation of quadratic extensions with the help of a certain cohomology class, which we will introduce in section \ref{MSLA:Kohom}.

There is a natural notion of the direct sum of quadratic extensions. That is, if
$(\mathfrak g_j,\mathfrak i_j,i_j,p_j)$, $j=1,2$ are quadratic extensions of $\mathfrak l_j$ by $\mathfrak a_j$, then 
$(\mathfrak g_1\oplus \mathfrak g_2,\mathfrak i_1\oplus \mathfrak i_2,i_1 \oplus i_2,p_1\oplus p_2)$ is a quadratic extension of $\mathfrak l_1\oplus\mathfrak l_2$ by $\mathfrak a_1\oplus\mathfrak a_2$.

We call a direct sum non-trivial, if both summands are different from the trivial quadratic extension $(\{0\},\{0\},0,0)$ of $\{0\}$ by $\{0\}$.
We call a quadratic extension decomposable, if it can be written as a non-trivial direct sum of two quadratic extensions.
A quadratic extension, which is equivalent to a decomposable one, is also decomposable.
Moreover, if a quadratic extension $(\mathfrak g,\mathfrak i,i,p)$ is decomposable, then the corresponding metric symplectic Lie algebra $\mathfrak g$ is decomposable as a metric symplectic Lie algebra.
Conversely, we have the following lemma from \cite{KO08} (see also \cite{KO06}), since it does not depend on the semisimplicity of the derivation.

\begin{lemma} (\cite{KO08}, see also \cite{KO06})
Let $(\mathfrak g,\mathfrak i,i,p)$ be a balanced quadratic extension of $\mathfrak l$ by $\mathfrak a$.
The quadratic extension $(\mathfrak g,\mathfrak i,i,p)$ is decomposable if and only if $\mathfrak g$ is decomposable as a metric symplectic Lie algebra.
\end{lemma}

\section{Quadratic cohomology}\label{MSLA:Kohom}
The task of this section is to introduce the cocycle, we will use to describe the quadratic extensions. Afterwards, we define the quadratic cohomology by using a certain group action on the set of cocycles.
\\\quad\\
Let $\rho:\mathfrak{l}\rightarrow\Hom(\mathfrak{a})$ be a representation of the Lie algebra $\mathfrak{l}$ on the vector space $\mathfrak{a}$.
Let $C^p(\mathfrak{l},\mathfrak{a})=\Hom(\bigwedge^p\mathfrak{l},\mathfrak{a})$ denote the space of alternating $p$-linear maps of $\mathfrak l$ with values in $\mathfrak a$ and
$(C^*(\mathfrak{l},\mathfrak{a}),\rd)$ the standard Lie algebra cochain complex, where $\rd:C^p(\mathfrak{l},\mathfrak{a})\rightarrow C^{p+1}(\mathfrak{l},\mathfrak{a})$ is defined by 
\begin{align*}
\rd\tau(L_1,\dots L_{p+1}) =&\sum_{i=1}^{p+1}(-1)^{i+1}\rho(L_i)\tau(L_1,\dots,\hat{L_i},\dots,L_{p+1})\\
&+\sum_{i<j}(-1)^{i+j}\tau([L_i,L_j],L_1,\dots,\hat{L_i},\dots,\hat{L_j},\dots,L_{p+1})
\end{align*}
for all $\tau\in C^p(\mathfrak{l},\mathfrak{a})$.
Moreover, let $Z^p(\mathfrak{l},\mathfrak{a})$ and $B^p(\mathfrak{l},\mathfrak{a})$ denote the groups of cocycles and coboundaries of $C^p(\mathfrak{l},\mathfrak{a})$.
The cochain complex of $\mathfrak{l}$ with the trivial representation on $\mathds R$ is denoted by $C^*(\mathfrak l)$.

\begin{definition}
Let $(\rho,\mathfrak{a},\langle\cdot,\cdot\rangle)$ be an orthogonal $\mathfrak{l}$-module.
We define a bilinear multiplication $\langle\cdot\wedge\cdot\rangle:C^p(\mathfrak l,\mathfrak a)\times C^q(\mathfrak l,\mathfrak a)\rightarrow C^{p+q}(\mathfrak l,\mathds R)$ by \[\langle\alpha\wedge\tau\rangle(L_1,\dots,L_{p+q})=\sum_{[\sigma]\in\mathcal{S}_{p+q}/\mathcal{S}_{p}\times\mathcal{S}_{q}}sgn(\sigma)\langle\alpha(L_{\sigma(1)},\dots,L_{\sigma(p)}),\tau(L_{\sigma(p+1)},\dots,L_{\sigma(p+q)})\rangle.\] 
Here $\mathcal S_k$ denotes the symmetric group of $k$ letters.
\end{definition}

\begin{lemma} (\cite[page 90]{KO06})
Assume $\alpha\in C^p(\mathfrak l,\mathfrak a)$ and $\tau\in C^q(\mathfrak l,\mathfrak a)$. Then
\begin{align}
\rd\langle\alpha\wedge\tau\rangle&=\langle \rd\alpha\wedge\tau\rangle+(-1)^p\langle\alpha\wedge \rd\tau\rangle \text{ and }\label{dWedge}\\
\langle\alpha\wedge\tau\rangle&=(-1)^{pq}\langle\tau\wedge\alpha\rangle.\label{AltWedge}
\end{align}
\end{lemma}

We consider the pairs $(\mathfrak l,D_\mathfrak l,\mathfrak a)$ of Lie algebras  with derivations $(\mathfrak l,D_\mathfrak l)$ and orthogonal $(\mathfrak l,D_\mathfrak l)$-modules $(\mathfrak a,D_\mathfrak a)$.
This pairs form a category, whose morphisms are pairs $(S,U)$ containing an 
homomorphism $S:\mathfrak l_1 \rightarrow \mathfrak l_2$ of Lie algebras and an isometric embedding $U:\mathfrak a_2 \rightarrow \mathfrak a_1$ satisfying $SD_{\mathfrak l_1}=D_{\mathfrak l_2}S$, $UD_{\mathfrak a_2}=D_{\mathfrak a_1}U$ and $U\circ\rho_2(SL)=\rho_1(L)\circ U$ for all $L\in\mathfrak{l}_1$.

We will denote the morphisms of this category by morphisms of pairs.

\begin{remark}
If $(S,U)$ is a morphism from $(\mathfrak l_1,D_{\mathfrak{l}_1},\mathfrak{a}_1)$ to $(\mathfrak l_2,D_{\mathfrak{l}_2},\mathfrak{a}_2)$, so it is a morphism from $(\mathfrak l_1,{D_{\mathfrak{l}_1}}_s,\mathfrak{a}_1)$ to $(\mathfrak l_2,{D_{\mathfrak l_2}}_s,\mathfrak a_2)$,
where ${D_{\mathfrak l_i}}_s$ denotes the semisimple part of the Jordan decomposition of $D_{\mathfrak l_i}$ and $\mathfrak a_i=(\mathfrak a_i,{D_{\mathfrak a_i}}_s)$
the orthogonal $(\mathfrak l,{D_{\mathfrak{l}_i}}_s)$-module with the semisimple part ${D_{\mathfrak a_i}}_s$ of the Jordan decomposition of $D_{\mathfrak a_i}$ for $i=1,2$.
\end{remark}

\begin{remark}
The category of pairs of Lie algebras and orthogonal modules was already introduced in \cite{KO06}.
The morphisms of that category were also called morphisms of pairs.
In \cite{KO08}, morphisms of pairs were build, which also respect additional structure on the pairs of Lie algebras and orthogonal modules.
A morphism of pairs from $(\mathfrak l_1,{D_{\mathfrak{l}_1}}_s,\mathfrak{a}_1)$ to $(\mathfrak l_2,{D_{\mathfrak l_2}}_s,\mathfrak a_2)$ in the sense of our definition,
where ${D_{\mathfrak l_i}}_s$ is a semisimple derivation of $\mathfrak l_i$ and $\mathfrak a_i=(\mathfrak a_i,{D_{\mathfrak a_i}}_s)$ an orthogonal $(\mathfrak l_i,{D_{\mathfrak l_i}}_s)$-module with semisimple ${D_{\mathfrak a_i}}_s$ for $i=1,2$,
is exactly the special case of the definition of morphisms of pairs of $(\mathds R,\{e\})$-equivariant metric Lie algebras in \cite{KO08}.
\end{remark}

\begin{definition}
The direct sum of two pairs $(\mathfrak l_j,D_{\mathfrak l_j},\mathfrak a_j)$, $j=1,2$ is defined by
\[(\mathfrak l,D_\mathfrak l,\mathfrak a)=(\mathfrak l_1,D_{\mathfrak l_1},\mathfrak a_1)\oplus(\mathfrak l_2,D_{\mathfrak l_2},\mathfrak a_2):=(\mathfrak l_1\oplus \mathfrak l_2,D_{\mathfrak l_1}\oplus D_{\mathfrak l_2},\mathfrak a_1\oplus \mathfrak a_2),\]
where $\mathfrak a_1$ and $\mathfrak a_2$ are orthogonal to each other, $D_{\mathfrak a_1}\oplus D_{\mathfrak a_2}$ is the derivation on $\mathfrak a_1\oplus\mathfrak a_2$ and for $i\neq j$, $i,j=1,2$ the Lie algebra $\mathfrak l_i$ acts trivially on $\mathfrak a_j$.

We call a direct sum non-trivial, if both summands are different from the trivial pair $(0,0,0)$.
\end{definition}

Of course, if $(S,U):(\mathfrak{l}_1,D_{\mathfrak{l}_1},\mathfrak{a}_1)\rightarrow (\mathfrak{l}_2,D_{\mathfrak{l}_2},\mathfrak{a}_2)$ is an isomorphism of pairs and $(\mathfrak{l}_2,D_{\mathfrak{l}_2},\mathfrak{a}_2)$ a non-trivial direct sum of pairs, then $(\mathfrak{l}_1,D_{\mathfrak{l}_1},\mathfrak{a}_1)$ is also a non-trivial direct sum of pairs.

Let $(S,U):(\mathfrak{l}_1,D_{\mathfrak{l}_1},\mathfrak{a}_1)\rightarrow (\mathfrak{l}_2,D_{\mathfrak{l}_2},\mathfrak{a}_2)$ be a morphism of pairs. We define the following pull back maps
\[(S,U)^*: C^p(\mathfrak{l}_2,\mathfrak{a}_2)\rightarrow C^p(\mathfrak{l}_1,\mathfrak{a}_1),\quad (S,U)^*\alpha(L_1,\dots,L_p):=U\circ\alpha(S(L_1),\dots,S(L_p))\]
\[\text{and } (S,U)^*: C^p(\mathfrak{l}_2)\rightarrow C^p(\mathfrak{l}_1),\quad (S,U)^*\gamma(L_1,\dots,L_p):=\gamma(S(L_1),\dots,S(L_p)).\]

\begin{lemma} (\cite[page 92]{KO06})
The pull backs $(S,U)^*$ commute with the differential $\rd$ and we have 
\begin{align}\label{eq:SLangleRangle}
(S,U)^*\langle\alpha\wedge\tau\rangle=\langle(S,U)^*\alpha\wedge(S,U)^*\tau\rangle
\end{align}
for $\alpha\in C^p(\mathfrak l_2,\mathfrak a_2)$ and $\tau\in C^q(\mathfrak l_2,\mathfrak a_2)$.
\end{lemma}

Let $\mathfrak{l}$ be a Lie algebra, $D_\mathfrak l$ a derivation of $\mathfrak l$ and $(\mathfrak{a},D_\mathfrak a)$ an orthogonal $(\mathfrak{l},D_\mathfrak l)$-module. Then $(\e^{-tD_\mathfrak l},\e^{tD_\mathfrak a})$ is an isomorphism of pairs for every $t\in\mathds{R}$.
For $\alpha\in C^p(\mathfrak{l},\mathfrak{a})$ denote $D^\circ\alpha=\frac{\rd}{\rd t}(\e^{-tD_\mathfrak l},\e^{tD_\mathfrak a})^*\alpha \big|_{t=0}$. We obtain 
\[D^\circ\alpha(L_1,\dots,L_p)=D_\mathfrak a(\alpha(L_1,\dots,L_p))-\sum_{i=1}^p\alpha(L_1,\dots,D_\mathfrak l{L_i},\dots,L_p).\]

For $\gamma\in C^p(\mathfrak l)$ we get analogous
\[D^\circ\gamma(L_1,\dots,L_p)=-\sum_{i=1}^p\gamma(L_1,\dots,D_\mathfrak l{L_i},\dots,L_p).\]

\begin{lemma}
Suppose $\alpha\in C^p(\mathfrak l,\mathfrak a)$, $\tau\in C^q(\mathfrak l,\mathfrak a)$ and $\gamma\in C^p(\mathfrak l)$. Then we have
\begin{align}
D^\circ\langle\alpha\wedge\tau\rangle&=\langle D^\circ\alpha\wedge\tau\rangle+\langle\alpha\wedge D^\circ\tau\rangle,\label{DlWedge}\\
D^\circ \rd\gamma&=\rd D^\circ\gamma\label{dDlSigma},\\
D^\circ \rd\alpha&=\rd D^\circ\alpha\label{dDaTau}.
\end{align}
\end{lemma}

\begin{proof}
This follows from the properties of $(S,U)^*$.
\end{proof}

\begin{remark} \label{rm:SD=DS}
The pull back maps and $D^\circ$ commute, i. e. $D^\circ_{1}(S,U)^*\gamma=(S,U)^*D^\circ_{2}\gamma$ and $D_1^\circ(S,U)^*\alpha=(S,U)^*D_2^\circ\alpha$.
\end{remark}

Now, we give a cohomology, which also respects the (not necessary semisimple) derivations.

Let ${D_\mathfrak l}_s$ and ${D_\mathfrak a}_s$ denote the semisimple parts of the Jordan decomposition of the derivations $D_\mathfrak l$ of $\mathfrak l$ and $D_\mathfrak a$ of $\mathfrak a$ for the rest of this work.
The nilpotent parts are denoted by ${D_\mathfrak l}_n$ and ${D_\mathfrak a}_n$.
Moreover, denote $D^\circ_s\alpha=\frac{\rd}{\rd t}(\e^{-t{D_\mathfrak l}_s},\e^{t{D_\mathfrak a}_s})^*\alpha \big|_{t=0}$.

\begin{definition}
Let $p$ be even. We set (exactly as in \cite{KO08})
\begin{align}
Z_{Q}^p(\mathfrak l,\phi_\mathfrak l,\mathfrak a):=\big\{(\alpha,\gamma)\in &C^p(\mathfrak l,\mathfrak a)\oplus C^{2p-1}(\mathfrak l)\mid \rd\alpha=0,\rd\gamma=\frac{1}{2}\langle\alpha\wedge\alpha\rangle,D_s^\circ\alpha=0,D_s^\circ\gamma=0\big\}.
\end{align}
it is not hard to see that $(\alpha,\gamma)\in Z^p_Q(\mathfrak l,\phi_\mathfrak l,\mathfrak a)$ is invariant under the morphisms of pairs
\begin{align}\label{eq:hKequiv}
(\e^{t{D_\mathfrak l}_s},\e^{-t{D_\mathfrak a}_s}),\quad t\in\mathds R.
\end{align}

Let $C^{p-1}_Q(\mathfrak l,\phi_\mathfrak l,\mathfrak a)\subset C^{p-1}(\mathfrak l,\mathfrak a)\oplus C^{2p-2}(\mathfrak l)$ denote the set of tuples $(\tau,\sigma)$, which are invariant under the morphisms of pairs (\ref{eq:hKequiv}). In our notation, this means that $(\tau,\sigma)$ satisfies 
\begin{align}
D^\circ_s\tau=0\quad\text{ and }\quad D^\circ_s\sigma=0.
\end{align}
\end{definition}

\begin{lemma}(\cite[page 13]{KO08}, see also \cite[Definition 1.1, Lemma 1.2]{KO06})
The set $C_Q^{p-1}(\mathfrak l,\phi_\mathfrak l,\mathfrak a)$ becomes a group with group multiplication
\begin{align}
(\tau_1,\sigma_1)(\tau_2,\sigma_2):=(\tau_1+\tau_2,\sigma_1+\sigma_2+\frac{1}{2}\langle\tau_1\wedge\tau_2\rangle).
\end{align}
Moreover, suppose $(\alpha,\gamma)\in C^p(\mathfrak l,\mathfrak a)\oplus C^{2p-1}(\mathfrak l)$ and $(\tau,\sigma)\in C_Q^{p-1}(\mathfrak l,\phi_\mathfrak l,\mathfrak a)$. Then
\begin{align}\label{Rechtswirkung1}
(\alpha,\gamma)(\tau,\sigma):=\big(\alpha+\rd\tau,\gamma+\rd\sigma+\langle(\alpha+\frac{1}{2}\rd\tau)\wedge\tau\rangle\big)
\end{align}
defines a right action of the group $C_Q^{p-1}(\mathfrak l,\phi_\mathfrak l,\mathfrak a)$ on $C^p(\mathfrak l,\mathfrak a)\oplus C^{2p-1}(\mathfrak l)$, which leaves $Z_{Q}^{p}(\mathfrak l,\phi_\mathfrak l,\mathfrak a)$ invariant.
\end{lemma}

Now, we set the $p$-th. cohomology set $H_Q^p(\mathfrak l,\phi_\mathfrak l,\mathfrak a)$ of $\mathfrak l$ with coefficients in $\mathfrak a$ as 
\[H_Q^p(\mathfrak l,\phi_\mathfrak l,\mathfrak a):=Z_Q^{p}(\mathfrak l,\phi_\mathfrak l,\mathfrak a)/C_Q^{p-1}(\mathfrak l,\phi_\mathfrak l,\mathfrak a).\]
For $(\alpha,\gamma)\in Z^p_Q(\mathfrak{l},\phi_\mathfrak l,\mathfrak{a})$ let $[\alpha,\gamma]$ denote the corresponding cohomology class.

\begin{remark}
The cohomology set $H_Q^p(\mathfrak l,\phi_\mathfrak l,\mathfrak a)$ was already studied in \cite{KO08}.
It was introduced to classify the equivalence classes of quadratic extensions of $(\mathfrak h,K)$-equivariant metric Lie algebras.
Here this cohomology set is a special case and describes the equivalence classes of metric Lie algebras with semisimple skewsymmetric derivations.
We shall detail this later.
\end{remark}

\begin{lemma}
Let $p$ be even and $(\delta,\epsilon),(\tau,\sigma)\in C^{p-1}_Q(\mathfrak{l},\phi_\mathfrak l,\mathfrak{a})$. Then
\begin{align}\label{Rechtswirkung2}
(\delta,\epsilon)(\tau,\sigma):=\big(\delta+D^\circ\tau,\epsilon+D^\circ\sigma+\langle(\delta+\frac{1}{2}D^\circ\tau)\wedge\tau\rangle\big)
\end{align}
defines a right action of $C_Q^{p-1}(\mathfrak l,\phi_\mathfrak l,\mathfrak a)$ on $C_Q^{p-1}(\mathfrak l,\phi_\mathfrak l,\mathfrak a)$.
\end{lemma}

\begin{proof}
Because of equation (\ref{DlWedge}) we have
\begin{align*}
(\delta,\epsilon)\big((\tau_1,\sigma_1)(\tau_2,\sigma_2)\big)=\big((\delta,\epsilon)(\tau_1,\sigma_1)\big)(\tau_2,\sigma_2).
\end{align*}
Since ${D_s}^\circ D^\circ=D^\circ{D_s}^\circ$, we obtain $D_s^\circ(\delta+D^\circ\tau)=0$ and \[D_s^\circ(\epsilon+D^\circ\sigma+\langle(\delta+\frac{1}{2}D^\circ\tau)\wedge\tau\rangle)=0\] for $(\delta,\epsilon),(\tau,\sigma)\in C^{p-1}_Q(\mathfrak{l},\phi_\mathfrak l,\mathfrak{a})$.
\end{proof}

\begin{definition}
Let $Z_{Q+}^p(\mathfrak l,D_\mathfrak l,\mathfrak a)$ denote the set of all $(\alpha,\gamma,\delta,\epsilon)\in Z^p_Q(\mathfrak l,\phi_\mathfrak l,\mathfrak a)\oplus C_Q^{p-1}(\mathfrak l,\phi_\mathfrak l,\mathfrak a)$ satisfying
$\rd\delta=D^\circ\alpha$ and $\rd\epsilon=D^\circ\gamma-\langle\alpha\wedge\delta\rangle$.
\end{definition}

\begin{lemma}
Let $p$ be even, $(\alpha,\gamma,\delta,\epsilon)\in Z^p_Q(\mathfrak l,\phi_\mathfrak l,\mathfrak a)\oplus C^{p-1}_Q(\mathfrak l,\phi_\mathfrak l,\mathfrak a)$ and $(\tau,\sigma)\in C_Q^{p-1}(\mathfrak l,\phi_\mathfrak l,\mathfrak a)$. Then
\begin{align}\label{Rechtswirkung}
(\alpha&,\gamma,\delta,\epsilon)(\tau,\sigma):=\big((\alpha,\gamma)(\tau,\sigma),(\delta,\epsilon)(\tau,\sigma)\big)\\
=&\big(\alpha+\rd\tau,\gamma+\rd\sigma+\langle(\alpha+\frac{1}{2}\rd\tau)\wedge\tau\rangle,\delta+D^\circ\tau,
\epsilon+D^\circ\sigma+\langle(\delta+\frac{1}{2}D^\circ\tau)\wedge\tau\rangle\big)\nonumber
\end{align}
defines a right action of $C_Q^{p-1}(\mathfrak l,\phi_\mathfrak l,\mathfrak a)$ on $Z^p_Q(\mathfrak l,\phi_\mathfrak l,\mathfrak a)\oplus C^{p-1}_Q(\mathfrak l,\phi_\mathfrak l,\mathfrak a)$, which leaves $Z_{Q+}^{p}(\mathfrak l,D_\mathfrak{l},\mathfrak a)$ invariant. 
\end{lemma}

\begin{proof}
Suppose $(\alpha,\gamma,\delta,\epsilon)\in Z^p_{Q+}(\mathfrak l,D_\mathfrak l,\mathfrak a)$ and $(\tau,\sigma)\in C_Q^{p-1}(\mathfrak l,\phi_\mathfrak l,\mathfrak a)$.
Since the group action of $C^{p-1}_Q(\mathfrak l,\phi_\mathfrak l,\mathfrak a)$ leaves the cocycles $Z^p_Q(\mathfrak l,\phi_\mathfrak l,\mathfrak a)$ invariant, it remains to show that the equations
\begin{align*}
&\rd(\delta+D^\circ\tau)-D^\circ(\alpha+d\tau)=0\text{ and}\\
\langle(\alpha+\rd\tau)\wedge(\delta&+D^\circ\tau)\rangle+\rd(\epsilon+D^\circ\sigma+\langle\delta\wedge\tau\rangle+\frac{1}{2}\langle D^\circ\tau\wedge\tau\rangle)\\
&-D^\circ(\gamma+\rd\sigma+\langle\alpha\wedge\tau\rangle+\frac{1}{2}\langle\tau\wedge \rd\tau\rangle)=0
\end{align*}
hold.
The first equation follows directly from (\ref{dDaTau}). 
Because of (\ref{dWedge}), (\ref{DlWedge}) and the commutativity of $\rd$ and $D^\circ$ the second equation is equivalent to
\begin{align*}
\langle\alpha\wedge\delta\rangle+\rd\epsilon-D^\circ\gamma+\langle(\rd\delta-D^\circ\alpha)\wedge\tau\rangle=0.
\end{align*}
Since $(\alpha,\gamma,\delta,\epsilon)\in Z^p_{Q+}(\mathfrak l,D_\mathfrak l,\mathfrak a)$, this equation is satisfied.
\end{proof}

We set the $p$-th. quadratic cohomology $H_{Q+}^p(\mathfrak l,D_\mathfrak l,\mathfrak a)$ of $\mathfrak l$ with coefficients in $\mathfrak a$ as 
\[H_{Q+}^p(\mathfrak l,D_\mathfrak l,\mathfrak a):=Z_{Q+}^{p}(\mathfrak l,D_\mathfrak l,\mathfrak a)/C_{Q}^{p-1}(\mathfrak l,\phi_\mathfrak l,\mathfrak a).\]
In addition, we denote the cohomology class of $(\alpha,\gamma,\delta,\epsilon)\in Z^p_{Q+}(\mathfrak{l},D_\mathfrak{l},\mathfrak{a})$ by $[\alpha,\gamma,\delta,\epsilon]$.

Let $(\mathfrak l,D_\mathfrak l,\mathfrak a)=(\mathfrak l_1,D_{\mathfrak l_1},\mathfrak a_1)\oplus(\mathfrak l_2,D_{\mathfrak l_2},\mathfrak a_2)$ be the direct sum of two pairs.
Let $j_i:\mathfrak a_i\rightarrow \mathfrak a$ and $q_i:\mathfrak l\rightarrow\mathfrak l_i$ denote the canonical embeddings and projections for $i=1,2$.
The addition in $C^p(\mathfrak l,\mathfrak a)\oplus C^{2p-1}(\mathfrak l)\oplus C^{p-1}(\mathfrak l,\mathfrak a)\oplus C^{2p-2}(\mathfrak l)$ defines a map
\[+:((q_1,j_1)^*Z^p_{Q+}(\mathfrak l_1,D_{\mathfrak l_1},\mathfrak a_1))\times((q_2,j_2)^*Z^p_{Q+}(\mathfrak l_2,D_{\mathfrak l_2},\mathfrak a_2))\rightarrow Z^p_{Q+}(\mathfrak l,D_\mathfrak l,\mathfrak a).\]

Since the addition respects the group action, we have a natural injective map
\[+:((q_1,j_1)^*H^p_{Q+}(\mathfrak l_1,D_{\mathfrak l_1},\mathfrak a_1))\times((q_2,j_2)^*H^p_{Q+}(\mathfrak l_2,D_{\mathfrak l_2},\mathfrak a_2))\rightarrow H^p_{Q+}(\mathfrak l,D_\mathfrak l,\mathfrak a).\]

We call a cohomology class $[\alpha,\gamma,\delta,\epsilon]\in H^p_{Q+}(\mathfrak l,D_\mathfrak l,\mathfrak a)$ decomposable if there is a decomposition $(\mathfrak l,D_\mathfrak l,\mathfrak a)=(\mathfrak l_1,D_{\mathfrak l_1},\mathfrak a_1)\oplus(\mathfrak l_2,D_{\mathfrak l_2},\mathfrak a_2)$ into a non-trivial direct sum of pairs such that
\[[\alpha,\gamma,\delta,\epsilon]\in (q_1,j_1)^*H^p_{Q+}(\mathfrak l_1,D_{\mathfrak l_1},\mathfrak a_1)+(q_2,j_2)^*H^p_{Q+}(\mathfrak l_2,D_{\mathfrak l_2},\mathfrak a_2).\]
Otherwise, the cohomology class is called indecomposable.
We denote the set of all indecomposable cohomology classes by $H^p_{Q+}(\mathfrak l,D_\mathfrak l,\mathfrak a)_i$.

\begin{lemma}(\cite[page 13]{KO08}, see also \cite[Seite 92]{KO06})\label{lm:4}
Let $(S,U):(\mathfrak l_1,{D_{\mathfrak l_1}}_s,\mathfrak a_1)\rightarrow(\mathfrak l_2,{D_{\mathfrak l_2}}_s,\mathfrak a_2)$ be a morphism of pairs, where ${D_{\mathfrak l_i}}_s$ is semisimple and $(\mathfrak a_i,{D_{\mathfrak a_i}}_s)$ orthogonal $(\mathfrak l_i,{D_{\mathfrak l_i}}_s)$-modules with semisimple derivations ${D_{\mathfrak a_i}}_s$ for $i=1,2$.
Then 
\[(S,U)^*(\alpha_2,\gamma_2):=\big((S,U)^*\alpha_2,(S,U)^*\gamma_2\big)\] defines a pull back map from $Z^p_Q(\mathfrak l_2,\phi_{\mathfrak l_2},\mathfrak a_2)$ to $Z^p_Q(\mathfrak l_1,\phi_{\mathfrak l_1},\mathfrak a_1)$.
\end{lemma}

\begin{lemma}
Let $(S,U):(\mathfrak l_1,D_{\mathfrak l_1},\mathfrak a_1)\rightarrow(\mathfrak l_2,D_{\mathfrak l_2},\mathfrak a_2)$ be a morphism of pairs.
Then
\begin{align}\label{eq:zusu}
(S,U)^*(\alpha_2,\gamma_2,\delta_2,\epsilon_2):=\big((S,U)^*\alpha_2,(S,U)^*\gamma_2,(S,U)^*\delta_2,(S,U)^*\epsilon_2\big)
\end{align}
defines a pull back map from 
$Z^p_{Q+}(\mathfrak l_2,D_{\mathfrak l_2},\mathfrak a_2)$ to $Z^p_{Q+}(\mathfrak l_1,D_{\mathfrak l_1},\mathfrak a_1)$.
\end{lemma}

\begin{proof}
From Lemma \ref{lm:4} we know that equation (\ref{eq:zusu}) defines a pull back from 
$Z^p_Q(\mathfrak l_2,\phi_{\mathfrak l_2},\mathfrak a_2)\oplus C^{p-1}_Q(\mathfrak l_2,\phi_{\mathfrak l_2},\mathfrak a_2)$ to $Z^p_Q(\mathfrak l_1,\phi_{\mathfrak l_1},\mathfrak a_1)\oplus C^{p-1}_Q(\mathfrak l_1,\phi_{\mathfrak l_1},\mathfrak a_1)$.
Because of equation (\ref{eq:SLangleRangle}), (\ref{dDlSigma}),(\ref{dDaTau}) and remark \ref{rm:SD=DS}, we get 
\[\rd(S,U)^*\delta_2-D^\circ(S,U)^*\alpha_2=0\] and \[\langle(S,U)^*\alpha_2\wedge(S,U)^*\delta_2\rangle+\rd(S,U)^*\epsilon_2-D^\circ (S,U)^*\gamma_2=0.\]
Thus $(S,U)^*$ defines a pull back map from $Z^p_{Q+}(\mathfrak l_2,D_{\mathfrak l_2},\mathfrak a_2)$ to $Z^p_{Q+}(\mathfrak l_1,D_{\mathfrak l_1},\mathfrak a_1)$.
\end{proof}

\begin{definition}
Since the pull back map respects the group action in the sense 
\[(S,U)^*((\alpha_2,\gamma_2,\delta_2,\epsilon_2)(\tau,\sigma))=((S,U)^*(\alpha_2,\gamma_2,\delta_2,\epsilon_2))((S,U)^*\tau,(S,U)^*\sigma),\]
we also have a pull back map from $H^p_{Q+}(\mathfrak l_2,D_{\mathfrak{l}_2},\mathfrak a_2)$ to $H^p_{Q+}(\mathfrak l_1,D_{\mathfrak{l}_1},\mathfrak a_1)$ given by
\[(S,U)^*[\alpha_2,\gamma_2,\delta_2,\epsilon_2]:=[(S,U)^*(\alpha_2,\gamma_2,\delta_2,\epsilon_2)].\]
\end{definition}

Of course, the pull back map of isomorphisms of pairs also map decomposable cohomology classes to decomposable ones.

\section{The standard model}\label{MSLA:Std}

In this section we define the canonical example $(\mathfrak d_{\alpha,\gamma},D_{\delta,\epsilon})$ of a metric Lie algebra with skewsymmetric derivation for every $(\alpha,\gamma,\delta,\epsilon)\in Z^2_{Q+}(\mathfrak l,D_\mathfrak l,\mathfrak a)$.
Moreover, we show that this standard model is also a standard example of a quadratic extension and we describe this balanced quadratic extensions on the level of the quadratic cocycles. Therefor, we define the subset $Z^2_{Q+}(\mathfrak l,D_\mathfrak l,\mathfrak a)_b$ of $Z^2_{Q+}(\mathfrak l,D_\mathfrak l,\mathfrak a)$.

\begin{example}
Let $(\mathfrak{l},[\cdot,\cdot]_\mathfrak l)$ be a Lie algebra, $(\rho,\mathfrak{a},\langle\cdot,\cdot\rangle_\mathfrak a)$ an orthogonal $\mathfrak{l}$-module and $(\alpha,\gamma)\in C^2(\mathfrak l,\mathfrak a)\oplus C^{3}(\mathfrak l)$.
We define a symmetric bilinear form $\langle\cdot,\cdot\rangle$ on $\mathfrak l^*\oplus\mathfrak a\oplus\mathfrak l$ by 
\[\langle Z_1+A_1+L_1,Z_2+A_2+L_2\rangle=Z_1(L_2)+\langle A_1,A_2\rangle_\mathfrak a+Z_2(L_1)\] for all $Z_1,Z_2\in\mathfrak l^*$, $A_1,A_2\in\mathfrak a$ and $L_1,L_2\in\mathfrak l$.
Moreover, we consider a skewsymmetric bilinear map $[\cdot,\cdot]:\mathfrak l^*\oplus\mathfrak a\oplus\mathfrak l\times\mathfrak l^*\oplus\mathfrak a\oplus\mathfrak l\rightarrow\mathfrak l^*\oplus\mathfrak a\oplus\mathfrak l$, which is given by
\begin{align*}
[\mathfrak l^*,\mathfrak l^*\oplus\mathfrak a]&=0,\\
[L_1,L_2]&=\gamma(L_1,L_2,\cdot)+\alpha(L_1,L_2)+[L_1,L_2]_\mathfrak l,\\
[L,A]&=-\langle A,\alpha(L,\cdot)\rangle+\rho(L)A,\\
[A_1,A_2]&=\langle\rho(\cdot)A_1,A_2\rangle,\\
[L,Z]&=\operatorname{ad}^*(L)(Z)=-Z([L,\cdot]_\mathfrak l)
\end{align*}
for all $Z\in\mathfrak l^*$, $A,A_1,A_2\in\mathfrak a$ and $L,L_1,L_2\in\mathfrak l$.

This definition is exactly the definition of the standard model in \cite{KO06}.
\end{example}

\begin{lemma}(\cite[page 9]{KO08}) \label{lm:2}
Let $(\mathfrak{l},{D_\mathfrak l}_s)$ be a Lie algebra with semisimple derivation ${D_\mathfrak l}_s$ and $(\rho,\mathfrak{a},\langle\cdot,\cdot\rangle,{D_\mathfrak a}_s)$ an orthogonal $(\mathfrak{l},{D_\mathfrak l}_s)$-module, where ${D_\mathfrak{a}}_s$ is semisimple.\\
Then $\mathfrak d_{\alpha,\gamma}(\mathfrak l,\mathfrak a):=(\mathfrak{l}^*\oplus\mathfrak{a}\oplus\mathfrak{l},[\cdot,\cdot],\langle\cdot,\cdot\rangle)$ is a metric Lie algebra with skewsymmetric derivation $D_{0,0}({D_\mathfrak l}_s,{D_\mathfrak a}_s):=-{D_\mathfrak{l}}_s^*\oplus  {D_\mathfrak{a}}_s\oplus{D_\mathfrak{l}}_s$ if and only if
$(\alpha,\gamma)\in Z_{Q}^2(\mathfrak l,\phi_\mathfrak l,\mathfrak a)$.
\end{lemma}

If $\mathfrak{l}$ and $\mathfrak{a}$ are clear from the context, we simply write $\mathfrak{d}_{\alpha,\gamma}$ for $\mathfrak d_{\alpha,\gamma}(\mathfrak l,\mathfrak a)$.

\begin{example}
Let $D_\mathfrak l$ and $D_\mathfrak a$ be derivations of $\mathfrak l$ and $\mathfrak a$. Consider $(\delta,\epsilon)\in C^{1}(\mathfrak l,\mathfrak a)\oplus C^{2}(\mathfrak l)$, we define
\[D_{\delta,\epsilon}(D_\mathfrak l,D_\mathfrak a)=\begin{pmatrix}
   -D_\mathfrak l^* & -\delta^* & \overline{\epsilon}\\
     0    &   D_\mathfrak a  & \delta\\
     0    &    0   & D_\mathfrak l
  \end{pmatrix}.\]
Here $\overline{\epsilon}$ is the uniquely determined linear map from $\mathfrak{l}$ to $\mathfrak{l}^*$ given by $\epsilon(L_1,L_2)=\langle \overline{\epsilon}(L_1),L_2\rangle=(\overline{\epsilon}(L_1))(L_2)$ for $L_1,L_2\in\mathfrak l$ and $\delta^*:\mathfrak a\rightarrow \mathfrak l^*$ is the dual map of $\delta$ given by $\langle\delta^*A,L\rangle=\langle A,\delta L\rangle$ for $L\in\mathfrak l$ and $A\in\mathfrak a$.
If $D_\mathfrak l$ and $D_\mathfrak a$ are clear from the context, we simply write $D_{\delta,\epsilon}$.
\end{example}

Let $i:\mathfrak a\rightarrow \mathfrak a\oplus \mathfrak l$ denote the canonical embedding and $p:\mathfrak a\oplus \mathfrak l\rightarrow \mathfrak l$ the canonical projection.
If $(\mathfrak{d}_{\alpha,\gamma}(\mathfrak{l},\mathfrak{a}),D_{\delta,\epsilon}(D_\mathfrak{l},D_\mathfrak{a}))$ is a metric Lie algebra with skewsymmetric derivation, then
$((\mathfrak d_{\alpha,\gamma}(\mathfrak l,\mathfrak a),D_{\delta,\epsilon}(D_\mathfrak l,D_\mathfrak a)),\mathfrak{l}^*,i,p)$ is a quadratic extension of $(\mathfrak l,D_\mathfrak l)$ by $(\mathfrak a,D_\mathfrak a)$, where we identify $\mathfrak d_{\alpha,\gamma}(\mathfrak l,\mathfrak a)/\mathfrak l^*$ and $\mathfrak a\oplus\mathfrak l$.

Let $(\mathfrak{d}_{\alpha,\gamma}(\mathfrak{l},\mathfrak{a}),D_{\delta,\epsilon}(D_\mathfrak{l},D_\mathfrak{a}))$  also denote the quadratic extension of the standard model.

\begin{theorem}\label{thm:3}
Let $(\mathfrak{l},D_\mathfrak l)$ be a Lie algebra with derivation $D_\mathfrak l$ and $(\rho,\mathfrak{a},\langle\cdot,\cdot\rangle,D_\mathfrak a)$ an orthogonal $(\mathfrak{l},D_\mathfrak l)$-module.
Then $(\mathfrak d_{\alpha,\gamma}(\mathfrak l,\mathfrak a),D_{\delta,\epsilon}(D_\mathfrak l,D_\mathfrak a))$ is a metric Lie algebra with skewsymmetric derivation, whose semisimple part of the Jordan decomposition of $D_{\delta,\epsilon}(D_\mathfrak l,D_\mathfrak a)$ is equal to $D_{0,0}({D_\mathfrak l}_s,{D_\mathfrak a}_s)$, if and only if $(\alpha,\gamma,\delta,\epsilon)\in Z_{Q+}^2(\mathfrak l,D_\mathfrak l,\mathfrak a)$.
\end{theorem}

\begin{proof}
Because of Lemma \ref{lm:2} the standard model $(\mathfrak d_{\alpha,\gamma}(\mathfrak l,\mathfrak a),D_{0,0}({D_\mathfrak l}_s,{D_\mathfrak a}_s))$ is a metric Lie algebra with skewsymmetric derivation,
if and only if $(\alpha,\gamma)$ is an element in $Z^2_{Q}(\mathfrak l,\phi_\mathfrak l,\mathfrak a)$.
Thus it remains to show that $D_{\delta,\epsilon}(D_\mathfrak l,D_\mathfrak a)$ is a skewsymmetric derivation of $\mathfrak d_{\alpha,\gamma}(\mathfrak l,\mathfrak a)$ if and only if the two conditions
\begin{align}
\rd\delta=D^\circ\alpha\label{eq:imB1},\\
\rd\epsilon=D^\circ\gamma-\langle\alpha\wedge \delta\rangle\label{eq:imB2}
\end{align}
hold and that $D_{0,0}({D_\mathfrak l}_s,{D_\mathfrak a}_s)$ is exactly the semisimple part of $D_{\delta,\epsilon}(D_\mathfrak l,D_\mathfrak a)$, if and only if $D_s^\circ\delta=0$ and $D_s^\circ\epsilon=0$ is satisfied.

So, let $D_{\delta,\epsilon}$ be a skewsymmetric derivation of $\mathfrak d_{\alpha,\gamma}$.
It holds $D[L_1,L_2]=[DL_1,L_2]+[L_1,DL_2]$ for $L_1,L_2\in\mathfrak l$ if and only if 
\begin{align*}
D_\mathfrak l[L_1,L_2]_\mathfrak l=&[D_\mathfrak lL_1,L_2]_\mathfrak l+[L_1,D_\mathfrak lL_2]_\mathfrak l,\\
D_\mathfrak a\alpha(L_1,L_2)=&-\delta[L_1,L_2]_\mathfrak l+\alpha(D_\mathfrak lL_1,L_2)-\rho(L_2)(\delta L_1)+\alpha(L_1,D_\mathfrak lL_2)+\rho(L_1)(\delta L_2),\\
\overline{\epsilon}[L_1,L_2]_\mathfrak l=&D_\mathfrak l^*\gamma(L_1,L_2,\cdot)+\gamma(D_\mathfrak lL_1,L_2,\cdot)+\langle \delta L_1,\alpha(L_2,\cdot)\rangle-\operatorname{ad}^*(L_2)(\overline{\epsilon}L_1)\\
&\delta^*\alpha(L_1,L_2)+\gamma(L_1,D_\mathfrak lL_2,\cdot)-\langle \delta L_2,\alpha(L_1,\cdot)\rangle+\operatorname{ad}^*(L_1)(\overline{\epsilon}L_2)
\end{align*}
is satisfied for all $L_1,L_2\in\mathfrak l$.
These equations are equivalent to (\ref{eq:imB1}) and (\ref{eq:imB2}).
On the other side it is easy to prove for $(\alpha,\gamma)\in Z^2_Q(\mathfrak l,\phi_\mathfrak l,\mathfrak a)$ that $D_{\delta,\epsilon}(D_\mathfrak l,D_\mathfrak a)$ is a skewsymmetric derivation of $\mathfrak d_{\alpha,\gamma}(\mathfrak l,\mathfrak a)$, because of
condition (\ref{eq:imB1}) and (\ref{eq:imB2}).
Finally, $D_{0,0}({D_\mathfrak l}_s,{D_\mathfrak a}_s)$ is the semisimple part of $D_{\delta,\epsilon}(D_\mathfrak l,D_\mathfrak a)$, if and only if 
\[D_{0,0}({D_\mathfrak l}_s,{D_\mathfrak a}_s)D_{\delta,\epsilon}({D_\mathfrak l}_n,{D_\mathfrak a}_n)=D_{\delta,\epsilon}({D_\mathfrak l}_n,{D_\mathfrak a}_n)D_{0,0}({D_\mathfrak l}_s,{D_\mathfrak a}_s)\]
holds, which is equivalent to $D_s^\circ\delta=0$ and $D_s^\circ\epsilon=0$.
\end{proof}

Again, we remark that nilpotent metric symplectic Lie algebra with bijective skewsymmetric derivations are our objects of interest to understand the metric, symplectic Lie algebras.
So, we already defined balanced quadratic extensions only for this kind of metric Lie algebras with derivations in section \ref{MSLA:QuadrErw} and now we again limit our observations for the rest of this section.

\begin{definition}
Let $\mathfrak l$ be a nilpotent Lie algebra with bijective derivation $D_\mathfrak l$
and $(\mathfrak a,D_\mathfrak a)$ a trivial $(\mathfrak l,D_\mathfrak l)$-module with bijective $D_\mathfrak a$. Let $m$ denote the smallest positive integer such that $\mathfrak l^{m+2}=0$.
We define the set of all balanced cocycles
$Z^2_{Q}(\mathfrak l,\phi_\mathfrak l,\mathfrak a)_b$
as the set of all cocycles $(\alpha,\gamma)\in Z^2_{Q}(\mathfrak l,\phi_\mathfrak l,\mathfrak a)$, which satisfy the following conditions for every $k=0,\dots,m$: 
\begin{itemize}
\item[($A_k$)] If there is an $A_0\in\mathfrak a$ and $Z_0\in (\mathfrak l^{k+1})^*$ for a given $L_0\in \mathfrak z(\mathfrak l)\cap\mathfrak l^{k+1}$ such that
\begin{align*}
&\alpha(L,L_0)=0,\\
&\gamma(L,L_0,\cdot)=-\langle A_0,\alpha(L,\cdot)\rangle_\mathfrak a +\langle Z_0,[L,\cdot]_\mathfrak l\rangle\text{ as an element of }(\mathfrak l^{k+1})^*
\end{align*}
is satisfied for all $L\in\mathfrak l$, then $L_0=0$.
\item[($B_k$)] The subspace $\alpha(\ker[\cdot,\cdot]_{\mathfrak l\otimes\mathfrak l^{k+1}})\subset \mathfrak a$ is non-degenerate with respect to $\langle\cdot,\cdot\rangle_\mathfrak a$.
\end{itemize}
Moreover, let $Z^2_{Q+}(\mathfrak l,D_\mathfrak l,\mathfrak a)_b$ denote the set of cocycles $(\alpha,\gamma,\delta,\epsilon)\in Z^2_{Q+}(\mathfrak l,D_\mathfrak l,\mathfrak a)$, where $(\alpha,\gamma)\in Z^2_{Q+}(\mathfrak l,\phi_\mathfrak l,\mathfrak a)_b$.
\end{definition}

The set $Z^2_{Q}(\mathfrak l,\phi_\mathfrak l,\mathfrak a)_b$ is used to describe on the level of quadratic cocycles when standard models of nilpotent metric Lie algebras define balanced quadratic extensions. Here we have:

\begin{lemma} (\cite{KO08}, see also \cite{Ka07})
Let $(\mathfrak d_{\alpha,\gamma}(\mathfrak l,\mathfrak a),D_{0,0}({D_\mathfrak l}_s,{D_\mathfrak a}_s))$ be a nilpotent, metric Lie algebra with bijective derivation.
The quadratic extension $(\mathfrak d_{\alpha,\gamma}(\mathfrak l,\mathfrak a),D_{0,0}({D_\mathfrak l}_s,{D_\mathfrak a}_s))$ is balanced if and only if $(\alpha,\gamma)\in Z^2_{Q}(\mathfrak l,\phi_\mathfrak l,\mathfrak a)_b$.
\end{lemma}

The property that a quadratic extension is balanced is a property of the corresponding metric Lie algebra and does not depend on the derivation.
Thus, we obtain the following lemma:

\begin{lemma}\label{lm:30}
Let $(\mathfrak d_{\alpha,\gamma}(\mathfrak l,\mathfrak a),D_{\delta,\epsilon}(D_\mathfrak l,D_\mathfrak a))$ be a nilpotent, metric Lie algebra with bijective derivation and let $(\alpha,\gamma,\delta,\epsilon)\in Z^2_{Q+}(\mathfrak l,D_\mathfrak l,\mathfrak a)$ be given.
The quadratic extension $(\mathfrak d_{\alpha,\gamma}(\mathfrak l,\mathfrak a),D_{\delta,\epsilon}(D_\mathfrak l,D_\mathfrak a))$ of a nilpotent Lie algebra $\mathfrak{l}$ with bijective derivation $D_\mathfrak l$ by a trivial $(\mathfrak l,D_\mathfrak l)$-module $(\mathfrak a,D_\mathfrak a)$ with bijective $D_\mathfrak a$ is balanced if and only if $(\alpha,\gamma,\delta,\epsilon)$ is an element in $Z^2_{Q+}(\mathfrak l,D_\mathfrak l,\mathfrak a)_b$.
\end{lemma}

\section{Equivalence to the standard model}\label{MSLA:ÄquivStd}

The notation standard model for $(\mathfrak{d}_{\alpha,\gamma},D_{\delta,\epsilon})$ comes from the fact that every quadratic extension of metric Lie algebras with skewsymmetric derivations is equivalent to $(\mathfrak{d}_{\alpha,\gamma},D_{\delta,\epsilon})$ for a suitable $(\alpha,\gamma,\delta,\epsilon)\in Z^2_{Q+}(\mathfrak{l},D_\mathfrak{l},\mathfrak{a})$.
This, we will prove in this section.\\
\quad\\
Let $(\mathfrak g,D_s;\mathfrak i,i,p)$ be a quadratic extension of the metric Lie algebra with skewsymmetric semisimple derivation $(\mathfrak l,{D_\mathfrak l}_s)$ by an orthogonal $(\mathfrak{l},{D_\mathfrak{l}}_s)$-module $(\mathfrak a,{D_\mathfrak a}_s)$, where ${D_\mathfrak a}_s$ is semisimple.
The subspace $\mathfrak i$ is invariant under $D_s$, thus $\mathfrak i^\bot$ is also invariant.
Since $D_s$ is semisimple and $\mathfrak i$ isotropic, we can choose an isotropic complement of $\mathfrak i^\bot$, which is invariant under $D_s$.
Let $s:\mathfrak l\rightarrow V_\mathfrak l$ be a section with $\tilde p\circ s=\id$. Here $\tilde p:\mathfrak g\rightarrow\mathfrak l$ is the composition of the natural projection $\pi$ from $\mathfrak g$ to $\mathfrak g/\mathfrak i$ and $p$.
This section satisfies $D_s\circ s=s\circ{D_\mathfrak{l}}_s$ because of
\[ \tilde p\circ D_s\circ s=p\circ\pi\circ D_s\circ s=p\circ\overline D_s\circ\pi\circ s={D_\mathfrak l}_s\circ \tilde p\circ s=\tilde p\circ s\circ {D_\mathfrak l}_s.\]
Here we used that $\tilde p$ is a bijection between $\mathfrak l$ and $V_\mathfrak l$ and that $p$ is a homomorphism from $(\mathfrak g/\mathfrak i,\overline D_s)$ to $(\mathfrak l,{D_\mathfrak l}_s)$.
Let $V_\mathfrak a$ be an orthogonal complement of $\mathfrak i\oplus s(\mathfrak l)$ in $\mathfrak g$. We define $t:\mathfrak a\rightarrow V_\mathfrak a$ by
\[i(A)=t(A)+\mathfrak i\in \mathfrak g/\mathfrak i.\]
This map $t$ is an isometry, since $i:\mathfrak{a}\rightarrow \mathfrak{i^\bot/\mathfrak{i}}$ is one.
Moreover, we define an isomorphism $p^*:\mathfrak l^*\rightarrow\mathfrak i$ by
\[\langle p^*(Z),s(L)\rangle:=\langle Z,(\tilde p\circ s)(L)\rangle=Z(L)\]
for all $Z\in\mathfrak l^*$ and $L\in\mathfrak l$.
Then we define $\alpha\in C^2(\mathfrak l,\mathfrak a)$ and $\gamma\in C^3(\mathfrak l)$ by
\begin{align}
i(\alpha(L_1,L_2)):&=[sL_1,sL_2]-s[L_1,L_2]+\mathfrak i\in \mathfrak g/\mathfrak i,\label{eq:alphadef}\\
\gamma(L_1,L_2,L_3):&=\langle[s(L_1),s(L_2)],s(L_3)\rangle.\label{eq:gammadef}
\end{align}

\begin{proposition}(\cite{KO08}, see also \cite[Lemma 2.8 and proof]{KO06}) \label{prop:2}
It holds $(\alpha,\gamma)\in Z^2_{Q}(\mathfrak l,\phi_\mathfrak l,\mathfrak a)$.
Moreover, the map $\varphi:=p^*+t+s:\mathfrak{l}^*\oplus\mathfrak{a}\oplus\mathfrak{l}\rightarrow\mathfrak{g}$ is an equivalence of the quadratic extensions $(\mathfrak d_{\alpha,\gamma}(\mathfrak l,\mathfrak a),D_{0,0}({D_\mathfrak{l}}_s,{D_\mathfrak{a}}_s))$ and $(\mathfrak g, D_s;\mathfrak i, i,p)$.
\end{proposition}

\begin{proposition} \label{prop:existstndrdmdll}
Let $(\mathfrak g,D;\mathfrak i,i,p)$ be a quadratic extension of $(\mathfrak l,D_\mathfrak l)$ by $(\mathfrak a,D_\mathfrak a)$.
Then there is an $(\alpha,\gamma,\delta,\epsilon)\in Z^2_{Q+}(\mathfrak l,D_\mathfrak l,\mathfrak a)$ such that the quadratic extension $(\mathfrak d_{\alpha,\gamma}(\mathfrak l,\mathfrak a),D_{\delta,\epsilon})$ is equivalent to $(\mathfrak g, D;\mathfrak i, i,p)$.
In addition $(\alpha,\gamma,\delta,\epsilon)$ is an element in $Z^2_{Q+}(\mathfrak l,D_\mathfrak l,\mathfrak a)_b$ for an balanced quadratic extension.
\end{proposition}

\begin{proof}
We choose $(\alpha,\gamma)\in Z^2_Q(\mathfrak l,\phi_\mathfrak l,\mathfrak a)$ by equations (\ref{eq:alphadef}) and (\ref{eq:gammadef}) with respect to the quadratic extension $(\mathfrak{g},D_s;\mathfrak{i},i,p)$ of $(\mathfrak{l},{D_\mathfrak l}_s)$ by $(\mathfrak a,{D_\mathfrak a}_s)$.
Here $D_s$ denotes the semisimple part of the Jordan decomposition of the derivation $D$ (analogous ${D_\mathfrak{l}}_s$ and ${D_\mathfrak{a}}_s$ for $D_\mathfrak{l}$ and $D_\mathfrak{a}$).
Proposition \ref{prop:2} says that $\varphi:=p^*+t+s:\mathfrak l^*\oplus\mathfrak a\oplus\mathfrak l\rightarrow\mathfrak g$ is an equivalence of quadratic extensions with semisimple parts of the derivations of $(\mathfrak d_{\alpha,\gamma},D_{0,0}({D_\mathfrak l}_s,{D_\mathfrak a}_s),\mathfrak l^*,i,p)$ to $(\mathfrak g, D_s,\mathfrak i, i,p)$.
Now, we set
\begin{align}
\langle A,\delta(L)\rangle_\mathfrak a:&=\langle t(A),D(s(L))\rangle\label{eq:deltadef},\\
\epsilon(L_1,L_2):&=\langle D(s(L_1)),s(L_2)\rangle.\label{eq:epsilondef}
\end{align}
Since $\varphi$ is an isometry and $D$ and $D_{\delta,\epsilon}$ are skewsymmetric, we get that $\varphi\circ D_{\delta,\epsilon}=D\circ\varphi$ holds if and only if
\begin{align}
 \langle(D\circ\varphi)(L),X\rangle&= \langle(\varphi\circ D_{\delta,\epsilon})(L),X\rangle,\\
 \langle(D\circ\varphi)(A_1),t(A_2)\rangle&= \langle(\varphi\circ D_{\delta,\epsilon})(A_1),t(A_2)\rangle 
\end{align}
is satisfied for all $L\in\mathfrak{l}$, $X\in\mathfrak{g}$ and $A_1,A_2\in\mathfrak{a}$.
Because of 
\[(\tilde p\circ D\circ t)(A_1)=(p\circ\pi\circ D\circ t)(A_1)=(p\circ \overline D\circ i)(A_1)=(D_\mathfrak{l}\circ p\circ i)(A_1)=0\]
the element $(D\circ t)(A_1)$ lies in $V_\mathfrak{a}\oplus\mathfrak{i}=\mathfrak{i}^\bot$ for all $A_1\in\mathfrak{a}$ and we obtain
\begin{align*}
 \langle(D\circ\varphi)(A_1),t(A_2)\rangle&=\langle(D\circ t)(A_1),t(A_2)\rangle=\langle(\pi\circ D\circ t)(A_1),i(A_2)\rangle\\
 &=\langle(\overline D\circ i)(A_1),i(A_2)\rangle_{\mathfrak{g}/\mathfrak{i}}=\langle(i\circ D_\mathfrak{a})(A_1),i(A_2)\rangle_{\mathfrak{g}/\mathfrak{i}}\\
 &=\langle(t\circ D_\mathfrak{a})(A_1),t(A_2)\rangle=\langle(\varphi\circ D_{\delta,\epsilon})(A_1),t(A_2)\rangle.
\end{align*}
Here we used that $i$ is a homomorphism of Lie algebras with derivations from $(\mathfrak{a},D_\mathfrak{a})$ to $(\mathfrak{g}/\mathfrak{i},\overline D)$.
Using
\[(\tilde p\circ D\circ s)(L_1)=(p\circ\overline D\circ \pi\circ s)(L_1)=(D_\mathfrak{l}\circ\tilde p\circ s)(L_1)=D_\mathfrak{l}(L_1)\]
for all $L_1\in\mathfrak{l}$, we get furthermore
\begin{align*}
 \langle(\varphi\circ D_{\delta,\epsilon})(L_1),\,&p^*(Z_2)+t(A_2)+s(L_2)\rangle\\
 &= \langle p^*\overline \epsilon(L_1),s(L_2)\rangle+\langle(t\circ \delta)(L_1),t(A_2)\rangle+\langle(s\circ D_\mathfrak{l})(L_1),p^*(Z_2)\rangle\\ &=\epsilon(L_1,L_2)+\langle\delta(L_1),A_2\rangle_\mathfrak{a}+Z_2(D_\mathfrak{l}(L_1))\\
 &=\langle(D\circ s)(L_1),p^*(Z_2)+t(A_2)+s(L_2)\rangle
\end{align*}
for all $L_2\in\mathfrak{l}$, $A_2\in\mathfrak{a}$ and $Z_2\in\mathfrak{l}^*$.

Finally, $\varphi$ is an isomorphism from $(\mathfrak{d}_{\alpha,\gamma}(\mathfrak{l},\mathfrak{a}),D_{\delta,\epsilon}(D_\mathfrak{l},D_\mathfrak{a}))$ to $(\mathfrak{g},D)$, which maps \linebreak $D_{0,0}({D_\mathfrak{l}}_s,{D_\mathfrak{a}}_s)$ to the semisimple part of $D$.
Thus $(\mathfrak{d}_{\alpha,\gamma},D_{\delta,\epsilon}(D_\mathfrak{l},D_\mathfrak{a}))$ is a skewsymmetric derivation, whose semisimple part equals $D_{0,0}({D_\mathfrak l}_s,{D_\mathfrak a}_s)$.
So $(\alpha,\gamma,\delta,\epsilon)\in Z^2_{Q+}(\mathfrak{l},D_\mathfrak{l},\mathfrak{a})$ and $\varphi$ an equivalence of quadratic extensions, because of Theorem \ref{thm:3}.

If $(\mathfrak{g},\mathfrak{i},i,p)$ is balanced, then $(\mathfrak{d}_{\alpha,\gamma},D_{\delta,\epsilon})$ is balanced, because of $\varphi(\mathfrak{l}^*)=\mathfrak{i}$.
Using Lemma \ref{lm:30}, we obtain that the cocycle $(\alpha,\gamma,\delta,\epsilon)$ is balanced.
\end{proof}

\section{Equivalence classes of quadratic extensions}\label{MSLA:Äquivkl}
The task of this section is to prove the bijection between $H^2_{Q+}(\mathfrak l,D_\mathfrak l,\mathfrak a)$ and the equivalence classes of quadratic extensions of $(\mathfrak l,D_\mathfrak l)$ by $\mathfrak a$.

\begin{lemma}(\cite{KO08}, compare with the prove of Lemma 2.9 in \cite{KO06}) \label{lm:3}
Let $(\mathfrak{l},{D_\mathfrak l}_s)$ be a Lie algebra with semisimple derivation and $(\mathfrak{a},{D_\mathfrak a}_s)$ an orthogonal $(\mathfrak l,{D_\mathfrak{l}}_s)$-module, where ${D_\mathfrak{l}}_s$ is semisimple.
Assume $(\alpha_i,\gamma_i)\in Z^2_{Q+}(\mathfrak l,\phi_\mathfrak l,\mathfrak a)$ for $i=1,2$.
An isomorphism $F:\mathfrak d_{\alpha_1,\gamma_1}(\mathfrak{l},\mathfrak{a})\rightarrow\mathfrak d_{\alpha_2,\gamma_2}(\mathfrak{l},\mathfrak{a})$ of Lie algebras is an equivalence of the quadratic extensions $(\mathfrak d_{\alpha_i,\gamma_i}(\mathfrak l,\mathfrak a),D_{0,0}({D_\mathfrak l}_s,{D_\mathfrak a}_s))$ if and only if
\[F=\begin{pmatrix}\id&-\tau^*&\bar\sigma-\frac{1}{2}\tau^*\tau\\&\id&\tau\\&&\id\end{pmatrix}\]
for a $(\tau,\sigma)\in C^1(\mathfrak{l},\phi_\mathfrak{l},\mathfrak{a})$ satisfying $\sigma(\cdot,\cdot)=\langle\bar\sigma(\cdot),\cdot\rangle$ and $(\alpha_1,\gamma_1)=(\alpha_2,\gamma_2)(\tau,\sigma)$.
Here $\tau^*:\mathfrak a\rightarrow\mathfrak l^*$ denotes the dual map of $\tau$ defined by $\langle\tau^* A,L\rangle=\langle A,\tau L\rangle$ for $A\in\mathfrak a$ and $L\in\mathfrak l$.
\end{lemma}

\begin{lemma}\label{lemma:aequivstndrdmdll}
Let us suppose that $(\alpha_i,\gamma_i,\delta_i,\epsilon_i)\in Z^2_{Q+}(\mathfrak l,D_\mathfrak l,\mathfrak a)$ for $i=1,2$.
The quadratic extensions $(\mathfrak d_{\alpha_i,\gamma_i},D_{\delta_i,\epsilon_i})$ are equivalent if and only if
$[\alpha_1,\gamma_1,\delta_1,\epsilon_1]=[\alpha_2,\gamma_2,\delta_2,\epsilon_2]\in H^2_{Q+}(\mathfrak l,D_\mathfrak l,\mathfrak a)$.
\end{lemma}

\begin{proof}
Suppose $(\alpha_1,\gamma_1,\delta_1,\epsilon_1),(\alpha_2,\gamma_2,\delta_2,\epsilon_2)\in Z^2_{Q+}(\mathfrak l,D_\mathfrak l,\mathfrak a)$.
Then $(\mathfrak d_{\alpha_i,\gamma_i},D_{\delta_i,\epsilon_i})$ is a quadratic extension and $D_{0,0}({D_{\mathfrak l_i}}_s,{D_{\mathfrak a_i}}_s)$ the semisimple part of the Jordan decomposition of the derivation $D_{\delta_i,\epsilon_i}(D_{\mathfrak l_i},D_{\mathfrak a_i})$.

For a bijection $F$ from $(\mathfrak d_{\alpha_1,\gamma_1},D_{\delta_1,\epsilon_1})$ to $(\mathfrak d_{\alpha_2,\gamma_2},D_{\delta_2,\epsilon_2})$ the linear map $F\circ D_{0,0}({D_{\mathfrak l_1}}_s,{D_{\mathfrak a_1}}_s)\circ F^{-1}$ is the semisimple part of the Jordan decomposition of $D_{\delta_2,\epsilon_2}$.
Thus the map $F$ is an equivalence of $(\mathfrak d_{\alpha_i,\gamma_i},D_{\delta_i,\epsilon_i})$, $i=1,2$, if and only if it is an equivalence of the corresponding quadratic extensions from
$(\mathfrak d_{\alpha_1,\gamma_1}, D_{0,0}({D_\mathfrak l}_s,{D_\mathfrak a}_s))$ to $(\mathfrak{d}_{\alpha_2,\gamma_2},D_{0,0}({D_\mathfrak l}_s,{D_\mathfrak a}_s))$ 
and $F\circ D_{\delta_1,\epsilon_1}=D_{\delta_2,\epsilon_2}\circ F$ holds.
From Lemma \ref{lm:3} we obtain that the first equation is satisfied if and only if 
\[F=\begin{pmatrix}\id&-\tau^*&\bar\sigma-\frac{1}{2}\tau^*\tau\\&\id&\tau\\&&\id\end{pmatrix}\]
for a $(\tau,\sigma)\in C^1(\mathfrak{l},\phi_\mathfrak{l},\mathfrak{a})$, where  $\sigma(\cdot,\cdot)=\langle\bar\sigma(\cdot),\cdot\rangle$ and $(\alpha_1,\gamma_1)=(\alpha_2,\gamma_2)(\tau,\sigma)$.
Moreover, the condition
$D_{\delta_2,\epsilon_2}\circ F=F\circ D_{\delta_1,\epsilon_1}$ holds, if 
\begin{align}
\langle D_{\delta_2,\epsilon_2}\circ F L,A\rangle&=\langle F\circ D_{\delta_1,\epsilon_1} L,A\rangle\label{eq:DF=FD1}\\
\langle D_{\delta_2,\epsilon_2}\circ F L_1,L_2\rangle&=\langle F\circ D_{\delta_1,\epsilon_1} L_1,L_2\rangle\label{eq:DF=FD2}
\end{align}
is satisfied for all $L,L_1,L_2\in\mathfrak l$ and $A\in\mathfrak a$.
Here equation (\ref{eq:DF=FD1}) is equivalent to 
\begin{align}
\delta_1=\delta_2+D_\mathfrak a\circ \tau-\tau\circ D_\mathfrak l\label{eq:deltatauaequiv}.
\end{align}
Equation (\ref{eq:DF=FD2}) is equivalent to 
\begin{align*}
\epsilon_1(L_1,L_2)=&\,\epsilon_2(L_1,L_2)+D^\circ\sigma(L_1,L_2)+\langle\delta_1 L_1,\tau L_2\rangle\\
&-\langle\tau L_1,\delta_2 L_2\rangle+\frac{1}{2}\langle\tau\circ D_\mathfrak l L_1,\tau L_2\rangle+\frac{1}{2}\langle\tau L_1,\tau\circ D_\mathfrak l L_2\rangle,
\end{align*}
which leads to 
\[(\delta_1,\epsilon_1)=(\delta_2,\epsilon_2)(\tau,\sigma)\]
with the use of equation (\ref{eq:deltatauaequiv}).
\end{proof}

\begin{remark}
Using Lemma \ref{lm:30} and Lemma \ref{lemma:aequivstndrdmdll} we get that the group action of $C^1_Q(\mathfrak{l},\phi_\mathfrak{l},\mathfrak{a})$ on $Z^2_{Q+}(\mathfrak l,D_\mathfrak l,\mathfrak a)$ leaves the set of balanced cocycles $Z^2_{Q+}(\mathfrak l,D_\mathfrak l,\mathfrak a)_b$ invariant.
\end{remark}

For a nilpotent Lie algebra $\mathfrak l$ with bijective derivation $D_\mathfrak l$ and a trivial $(\mathfrak l,D_\mathfrak l)$-module $(\mathfrak a,D_\mathfrak a)$, where $D_\mathfrak a$ is bijective, we define
\begin{align}
H^2_{Q+}(\mathfrak l,D_\mathfrak l,\mathfrak a)_b:=Z^2_{Q+}(\mathfrak l,D_\mathfrak l,\mathfrak a)_b/C^1_{Q}(\mathfrak l,\phi_\mathfrak l,\mathfrak a).
\end{align}

\begin{theorem}
The Equivalence classes of quadratic extensions of $(\mathfrak{l},D_\mathfrak{l})$ by $(\mathfrak{a},D_\mathfrak{a})$ are in one-to-one correspondence with $H^2_{Q+}(\mathfrak{l},D_\mathfrak{l},\mathfrak{a})$.
Moreovere, the equivalence classes of balanced quadratic extensions are in bijection with $H^2_{Q+}(\mathfrak{l},D_\mathfrak{l},\mathfrak{a})_b$.
\end{theorem}

\begin{proof}
Let $(\mathfrak{g},D;\mathfrak{i},i,p)$ be a quadratic extension of $(\mathfrak{l},D_\mathfrak{l})$ by $(\mathfrak{a},D_\mathfrak{a})$.
Consider $(\alpha_i,\gamma_i,\delta_i,\epsilon_i)\in Z^2_{Q+}(\mathfrak{l},D_\mathfrak{l},\mathfrak{a})$, which are given by the equations (\ref{eq:alphadef}), (\ref{eq:gammadef}), (\ref{eq:deltadef}) and (\ref{eq:epsilondef}) with respect to two sections $s_i:\mathfrak{l}\rightarrow\mathfrak{g}$, $i=1,2$.
Proposition \ref{prop:existstndrdmdll} says that $(\mathfrak{d}_{\alpha_1,\gamma_1},D_{\delta_1,\epsilon_1})$ and $(\mathfrak{d}_{\alpha_2,\gamma_2},D_{\delta_2,\epsilon_2})$ are equivalent, since both are equivalent to $(\mathfrak{g},D;\mathfrak{i},i,p)$.
Thus $[\alpha_1,\gamma_1,\delta_1,\epsilon_1]=[\alpha_2,\gamma_2,\delta_2,\epsilon_2]$, because of Lemma \ref{lemma:aequivstndrdmdll}.
This shows that the cohomology group of $(\alpha,\gamma,\delta,\epsilon)\in Z^2_{Q+}(\mathfrak{l},D_\mathfrak{l},\mathfrak{a})$ defined by Proposition \ref{prop:existstndrdmdll} does not depend on the choice of the section $s$.
Using Theorem \ref{thm:3}, Proposition \ref{prop:existstndrdmdll} and Lemma \ref{lemma:aequivstndrdmdll} we obtain the assertion.
\end{proof}

\section{Isomorphism classes of metric, symplectic Lie algebras}\label{MSLA:Isomkl}
Until now, we know that we can represente the isomorphism classes of metric, symplectic Lie algebras by a standard model.
Now, we determine when two standard models are isomorphic as metric, symplectic Lie algebras and, finally, give a classification scheme for metric, symplectic Lie algebras.

\begin{proposition}(\cite{KO08}, see also the proof of Lemma 4.1 in \cite{KO06})\label{prop:3}
Let $(\mathfrak{l}_i,{D_{\mathfrak l_i}}_s)$ be Lie algebras with semisimple derivations ${D_{\mathfrak l_i}}_s$ and $(\mathfrak{a}_i,{D_{\mathfrak a_i}}_s)$ trivial $(\mathfrak{l}_i,{D_{\mathfrak l_i}}_s)$-modules with semisimple ${D_{\mathfrak a_i}}_s$. Suppose $(\alpha_i,\gamma_i)\in Z^2_{Q}(\mathfrak l_i,\phi_{\mathfrak l_i},\mathfrak a_i)_b$ for $i=1,2$.
Let \[F:(\mathfrak d_{\alpha_1,\gamma_1}(\mathfrak l_1,\mathfrak a_1),D_{\delta_1,\epsilon_1})\rightarrow (\mathfrak d_{\alpha_2,\gamma_2}(\mathfrak l_2,\mathfrak a_2),D_{\delta_2,\epsilon_2})\]
be an isomorphism and $\overline F:\mathfrak{a}_1\oplus\mathfrak{l}_1\rightarrow\mathfrak{a}_2\oplus\mathfrak{l}_2$ the corresponding isomorphism on the quotient induced by $F$.
We set 
\begin{align*}
S(L):=(p\circ \overline F)(L),\quad U(A):=\overline F^{\;-1}(A) 
\end{align*}
for $L\in\mathfrak l_1$, $A\in \mathfrak a_2$.
Then $(S,U)$ is an isomorphism of pairs of $(\mathfrak{l}_1,{D_{\mathfrak{l}_1}}_s,\mathfrak{a}_1)$ and $(\mathfrak{l}_2,{D_{\mathfrak{l}_2}}_s,\mathfrak{a}_2)$.
\end{proposition}

\begin{theorem}\label{ddIsomorph}
Suppose $(\alpha_i,\gamma_i,\delta_i,\epsilon_i)\in Z^2_{Q+}(\mathfrak l_i,D_{\mathfrak l_i},\mathfrak a_i)_b$ for $i=1,2$.
The metric Lie algebras with bijective skewsymmetric derivations $(\mathfrak d_{\alpha_i,\gamma_i}(\mathfrak l_i,\mathfrak a_i),D_{\delta_i,\epsilon_i})$ are isomorphic if and only if there is an isomorphism of pairs $(S,U):(\mathfrak l_1,D_{\mathfrak l_1},\mathfrak a_1)\rightarrow(\mathfrak l_2,D_{\mathfrak l_2},\mathfrak a_2)$ such that \[(S,U)^*[\alpha_2,\gamma_2,\delta_2,\epsilon_2]=[\alpha_1,\gamma_1,\delta_1,\epsilon_1]\in H^2_{Q+}(\mathfrak l_1,D_{\mathfrak l_1},\mathfrak a_1)_b.\]
\end{theorem}

\begin{proof}
Let $(S,U)$ be an isomorphism of pairs such that $(S,U)^*(\alpha_2,\gamma_2,\delta_2,\epsilon_2)$ lies in the cohomology set of $(\alpha_1,\gamma_1,\delta_1,\epsilon_1)$.
The map
\begin{align}\label{eq:IsoSU}
\varphi=\diag({S^*}^{-1},U^{-1},S):\mathfrak l^*_1\oplus\mathfrak a_1\oplus\mathfrak l_1\rightarrow\mathfrak l^*_2\oplus\mathfrak a_2\oplus\mathfrak l_2
\end{align}
defines an isomorphism of the Lie algebra $\mathfrak d_{(S,U)^*\alpha_2,(S,U)^*\gamma_2}(\mathfrak l_1,\mathfrak a_1)$ with corresponding derivation $D_{(S,U)^*\delta_2,(S,U)^*\epsilon_2}(D_{\mathfrak l_1},D_{\mathfrak a_1})$ to the Lie algebra $\mathfrak d_{\alpha_2,\gamma_2}$ with derivation $D_{\delta_2,\epsilon_2}$.

Since $(S,U)^*(\alpha_2,\gamma_2,\delta_2,\epsilon_2)$ and $(\alpha_1,\gamma_1,\delta_1,\epsilon_1)$ lie in the same cohomology set, the corresponding quadratic extensions are equivalent.
Especially $(\mathfrak d_{(S,U)^*\alpha_2,(S,U)^*\gamma_2},D_{(S,U)^*\delta_2,(S,U)^*\epsilon_2})$ is isomorphic to $(\mathfrak d_{\alpha_1,\gamma_1},D_{\delta_1,\epsilon_1})$.
Thus $(\mathfrak d_{\alpha_1,\gamma_1},D_{\delta_1,\epsilon_1})$ and $(\mathfrak d_{\alpha_2,\gamma_2},D_{\delta_2,\epsilon_2})$ are isomorphic.

Now, suppose $(\alpha_i,\gamma_i,\delta_i,\epsilon_i)\in Z^2_{Q+}(\mathfrak l_i,D_{\mathfrak l_i},\mathfrak a_i)_b$, $i=1,2$.
Let
\[F:(\mathfrak d_{\alpha_1,\gamma_1}(\mathfrak l_1,\mathfrak a_1),D_{\delta_1,\epsilon_1})\rightarrow (\mathfrak d_{\alpha_2,\gamma_2}(\mathfrak l_2,\mathfrak a_2),D_{\delta_2,\epsilon_2})\]
be an isomorphism of the metric Lie algebras.
Since the cocycles are balanced, we have $F(\mathfrak l_1^*)=\mathfrak l_2^*$ and $F$ induces an isomorphism  $\overline F:\mathfrak a_1\oplus\mathfrak l_1\rightarrow\mathfrak a_2\oplus\mathfrak l_2$. Now, we set 
\begin{align*}
S(L):=(p\circ \overline F)(L),\quad U(A):=\overline F^{\;-1}(A) 
\end{align*}
for $L\in\mathfrak l_1$, $A\in \mathfrak a_2$.
Because of $D_{\delta_1,\epsilon_1}\circ F=F\circ D_{\delta_2,\epsilon_2}$ and Proposition \ref{prop:3} we obtain that $(S,U)$ is an isomorphism of the pairs $(\mathfrak{l}_i,D_{\mathfrak l_i},\mathfrak{a}_i)$.
Thus, $\varphi$ given by equation (\ref{eq:IsoSU}) defines an isomorphism from the metric Lie algebra with skewsymmetric bijective derivation $(\mathfrak d_{(S,U)^*\alpha_2,(S,U)^*\gamma_2}(\mathfrak l_1,\mathfrak a_1),D_{(S,U)^*\delta_2,(S,U)^*\epsilon_2})$ to $(\mathfrak d_{\alpha_2,\gamma_2}(\mathfrak l_2,\mathfrak a_2),D_{\delta_2,\epsilon_2})$.
Thus the corresponding quadratic extensions $(\mathfrak d_{(S,U)^*\alpha_2,(S,U)^*\gamma_2}),D_{(S,U)^*\delta_2,(S,U)^*\epsilon_2},\mathfrak l_1^*,i,p)$ and $(\mathfrak d_{\alpha_1,\gamma_1},D_{\delta_1,\epsilon_1},\mathfrak l_1^*,i,p)$ are equivalent.
Using Lemma \ref{lemma:aequivstndrdmdll} yields the assertion.
\end{proof}

\begin{remark}
For an isomorphism of pairs $(S,U)$ the map $\varphi$ defined by equation (\ref{eq:IsoSU}) is an isomorphism of metric Lie algebras with derivations and it holds $\varphi(\mathfrak{l}_1^*)=\mathfrak{l}_2^*$.
Thus, using Lemma \ref{lm:30} we obtain that the isomorphisms of pairs define pull back maps, which map the balanced cocycles of $Z^2_{Q+}(\mathfrak{l}_2,D_{\mathfrak{l}_2},\mathfrak{a}_2)$ to balanced cocycles of $Z^2_{Q+}(\mathfrak{l}_1,D_{\mathfrak{l}_1},\mathfrak{a}_1)$. 
\end{remark}

\begin{lemma}\label{lm:decomp}
Assume $(\alpha,\gamma,\delta,\epsilon)\in Z^2_{Q+}(\mathfrak l,D_\mathfrak l,\mathfrak a)_b$. The quadratic extension 
$(\mathfrak d_{\alpha,\gamma}(\mathfrak l,\mathfrak a),D_{\delta,\epsilon}(D_\mathfrak l,D_\mathfrak a))$ is decomposable if and only if $[\alpha,\gamma,\delta,\epsilon]$ is decomposable.
\end{lemma}

\begin{proof}
In general, if a quadratic extension is decomposable, then the corresponding cohomology class $[\alpha,\gamma,\delta,\epsilon]$ given by equations (\ref{eq:alphadef}), (\ref{eq:gammadef}), (\ref{eq:deltadef}) and (\ref{eq:epsilondef}) equals 
\[(q_1,j_1)^*[\alpha^1,\gamma_1,\delta_1,\epsilon_1]+(q_2,j_2)^*[\alpha^2,\gamma_2,\delta_2,\epsilon_2],\]
where $[\alpha_i,\gamma_i,\delta_i,\epsilon_i]$ are the corresponding cohomology classes of the direct summands of the quadratic extension.
This immediately implies that $[\alpha,\gamma,\delta,\epsilon]\in H^2(\mathfrak l,D_\mathfrak l,\mathfrak a)$ is decomposable for an decomposable quadratic extension $(\mathfrak d_{\alpha,\gamma}(\mathfrak l,\mathfrak a),D_{\delta,\epsilon}(D_\mathfrak l,D_\mathfrak a))$.
Conversely, if $[\alpha,\gamma,\delta,\epsilon]\in H^2(\mathfrak l,D_\mathfrak l,\mathfrak a)_b$ is decomposable, then 
\[[\alpha,\gamma,\delta,\epsilon]=(q_1,j_1)^*[\alpha_1,\gamma_1,\delta_1,\epsilon_1]+(q_2,j_2)^*[\alpha_2,\gamma_2,\delta_2,\epsilon_2]\] for certain $[\alpha_i,\gamma_i,\delta_i,\epsilon_i]\in H^2_{Q+}(\mathfrak l_i,D_{\mathfrak l_i},\mathfrak a_i)$.
Then the observation in the beginning of the proof and Lemma \ref{lemma:aequivstndrdmdll} yield that 
$(\mathfrak d_{\alpha,\gamma}(\mathfrak l,\mathfrak a),D_{\delta,\epsilon}(D_\mathfrak l,D_\mathfrak a))$ is equivalent to the direct sum $(\mathfrak d_{\alpha_1,\gamma_1}(\mathfrak l_1,\mathfrak a_1),D_{\delta_1,\epsilon_1}(D_{\mathfrak l_1},D_{\mathfrak a_1}))\oplus (\mathfrak d_{\alpha_2,\gamma_2}(\mathfrak l_2,\mathfrak a_2),D_{\delta_2,\epsilon_2}(D_{\mathfrak l_2},D_{\mathfrak a_2}))$ and hence is decomposable.
\end{proof}

Let $G(\mathfrak l,D_\mathfrak l,\mathfrak a)$ denote the set of all morphisms of pairs $(S,U)$ constituting of automorphisms $S$ of $\mathfrak l$ and isometries $U$ of $\mathfrak a$, which satisfy $SD_\mathfrak l=D_\mathfrak lS$ and $UD_\mathfrak a=D_\mathfrak aU$.
We simply write $G$, if $(\mathfrak l,D_\mathfrak l,\mathfrak a)$ is clear from the context.
Let $H^2_{Q+}(\mathfrak l,D_\mathfrak l,\mathfrak a)_0$ denote the set of all balanced and indecomposable cohomology classes in $H^2_{Q+}(\mathfrak l,D_\mathfrak l,\mathfrak a)$.

The final result of this section is the following classification scheme. It follows directly from Lemma \ref{lm:msla1}, Theorem \ref{thm:1}, Proposition \ref{prop:existstndrdmdll} and Theorem \ref{ddIsomorph}.
\begin{theorem} \label{Klassifikation}
Let $\mathfrak l$ be a (nilpotent) Lie algebra, $D_\mathfrak l$ a bijective derivation of $\mathfrak l$ and $(\mathfrak a,D_\mathfrak a)$ a trivial $(\mathfrak l,D_\mathfrak l)$-module with skewsymmetric bijective derivation $D_\mathfrak a$ of $\mathfrak a$.
The set of isomorphism classes of indecomposable metric symplectic Lie algebras $\mathfrak g$ with the following properties
\begin{itemize}
\item $\mathfrak i(\mathfrak g)^\bot/\mathfrak i(\mathfrak g)$ together with the induced bijective skewsymmetric derivation is isomorphic to $(\mathfrak a,D_\mathfrak a)$ as a metric Lie algebra with derivation and
\item $\mathfrak g/\mathfrak i(\mathfrak g)^\bot$ together with the induced bijective derivation is isomorphic to $(\mathfrak l,D_\mathfrak l)$ as a Lie algebra with derivation
\end{itemize}
is in one-to-one correspondence to the $G(\mathfrak l,D_\mathfrak l,\mathfrak a)$-orbits of $H^2_{Q+}(\mathfrak l,D_\mathfrak l,\mathfrak a)_0$.

The set of isomorphism classes of indecomposable metric symplectic Lie algebras is in bijection to the set 
\begin{align*}
\coprod_{(\mathfrak l,D_\mathfrak l)\in \mathcal L}
\coprod_{\mathfrak a\in \mathcal A_{\mathfrak l,D_\mathfrak l}}
H^2_{Q+}(\mathfrak l,D_\mathfrak l,\mathfrak a)_0/G(\mathfrak l,D_\mathfrak l,\mathfrak a),
\end{align*}
where $\mathcal L$ is a system of representantives of the isomorphism classes of nilpotent Lie algebras with bijective derivations and $\mathcal A_{\mathfrak l,D_\mathfrak l}$ a system of representatives of the isomorphism classes of abelian metric Lie algebras with bijective skewsymmetric derivations (considered as trivial $(\mathfrak l,D_\mathfrak l)$-modules).
\end{theorem}

\begin{remark}
Here, the union was taken over a system of representatives $\mathcal L$ of the isomorphism classes of all Lie algebras $\mathfrak l$ with bijective derivations $D_\mathfrak l$ of $\mathfrak l$.
But often $(\mathfrak l,D_\mathfrak l)$ does not have a trivial $(\mathfrak l,D_\mathfrak l)$-module $\mathfrak a$ such that $H^2_{Q+}(\mathfrak l,D_\mathfrak l,\mathfrak a)_0$ is not empty.
We call a Lie algebra $\mathfrak l$ with bijective $D_\mathfrak l$ admissable if there is such a trivial $(\mathfrak l,D_\mathfrak l)$-module.
The set of isomorphism classes of indecomposable metric symplectic Lie algebras is in one-to-one correspondence to the set 
\begin{align*}
\coprod_{(\mathfrak l,D_\mathfrak l)\in \mathcal L_{ad}}
\coprod_{\mathfrak a\in \mathcal A_{\mathfrak l,D_\mathfrak l}}
H^2_{Q+}(\mathfrak l,D_\mathfrak l,\mathfrak a)_0/G(\mathfrak l,D_\mathfrak l,\mathfrak a),
\end{align*}
where $\mathcal L_{ad}$ is a system of representatives of the isomorphism classes of admissable Lie algebras with bijective derivations and $\mathcal A_{\mathfrak l,D_\mathfrak l}$ a system of representatives of isomorphism classes of abelian metric Lie algebras with bijective skewsymmetric derivations (considered as trivial $(\mathfrak l,D_\mathfrak l)$-modules).
\end{remark}

\section{Metric symplectic Lie algebras in special cases}\label{MSLA:Anw}

In this section we determine metric symplectic Lie algebras in spezial cases  up to isomorphism with the help of the classification scheme in section \ref{MSLA:Isomkl}.\\
\quad\\
We call a basis $\{a_1,\dots,a_p,a_{p+1},\dots,a_{p+q}\}$ of $\mathfrak{a}$ orthonormal, if $a_i\bot a_j$ for $i\neq j$ and $\langle a_i,a_i\rangle_\mathfrak{a}=-1$ for $i=1,\dots,p$ and $\langle a_j,a_j\rangle_\mathfrak{a}=1$ for $j=p+1,\dots,p+q$.
In this case $(p,q)$ is called signature and $p$ the index of $\langle\cdot,\cdot\rangle_\mathfrak{a}$.
Let $\langle\cdot,\cdot\rangle_{p,q}$ denote the non-degenerate symmetric bilinear form with signature $(p,q)$ on $\mathds R^{p+q}$, which satisfies that the standard basis of the $\mathds{R}^{p+q}$ is an orthonormal basis.
Then we call $\mathds R^{p,q}:=(\mathds R^{p+q},\langle\cdot,\cdot\rangle_{p,q})$ the standard pseudo-euclidian space.
Of course we will denote $\mathds R^{0,n}$ as $\mathds{R}^n$.
Often, we choose a Witt basis for $\mathds{R}^{1,1}$. This is a basis $\{a_1,a_2\}$ of $\mathds R^2$ of isotropic vectors, which satisfy $\langle a_1,a_2\rangle_{1,1}=1$.

For $z=a+ib\in\mathds{C}$ we set \[D_{z}:=\begin{pmatrix}a&-b\\b&a\end{pmatrix}.\]
Moreover, denote $D_{z_1,\dots,z_n}=\diag(D_{z_1},\dots,D_{z_n})$ for $z_1,\dots,z_n\in\mathds{C}$ the block diagonal matrix with the matrices $D_{z_1}$ till $D_{z_n}$ on the diagonal.
Denote \[N_b=\begin{pmatrix}b&1\\0&b\end{pmatrix}\] for $b\in\mathds{R}$.

For a basis $\{\sigma^1,\dots,\sigma^n\}$ of $\mathfrak l^*$ we write abbrevitory
$\sigma^{ij}:=\sigma^i\wedge\sigma^j$ and $\sigma^{ijk}:=\sigma^i\wedge\sigma^j\wedge\sigma^k$.

\begin{theorem}\label{thm:dim6}
The only non-abelian, metric Lie algebra of dimension smaller than eight having symplectic forms is $\mathfrak{d}_{0,\sigma^{123}}(\mathds{R}^3,0)$.
Its symplectic forms are given (up to isomorphism) by exactly one of the following derivations:
\[D_{0,0}(\diag(a,b,c),0),\quad
D_{0,0}(\diag(N_b,-2b),0),\quad
D_{0,0}(\diag(D_{b+id},-2b),0)\]
mit $a\leq b\leq c$, $a,b,c\neq 0$, $d>0$ und $a+b+c=0$.
\end{theorem}

\begin{theorem}\label{thm:dim8}
The set of $(\mathfrak d_{\alpha,\gamma}(\mathfrak l,\mathfrak a),D_{\delta,\epsilon}(D_\mathfrak l,D_\mathfrak a))$, where\\
$\mathfrak{l}=\mathds{R}^3$, $\mathfrak{a}\in\{\mathds{R}^2,\mathds{R}^{2,0}\}$, $D_\mathfrak{a}=D_{is}$ and 
\begin{itemize}
\item $D_\mathfrak{l}=\diag(D_{b+is},-b)$,
 $(\alpha,\gamma,\delta,\epsilon)=(\sigma^{13}\otimes a_1+\sigma^{23}\otimes a_2,0,0,0)$,
\end{itemize}
$\mathfrak{l}=\mathds{R}^3$, $\mathfrak{a}=\mathds{R}^{1,1}$, $D_\mathfrak{a}=\diag(s,-s)$ and
\begin{itemize}
\item $D_\mathfrak{l}=\diag(s-e,e,-s-e)$,
 $(\alpha,\gamma,\delta,\epsilon)=(\sigma^{12}\otimes a_1+\sigma^{23}\otimes a_2,0,0,0)$,
\item $D_\mathfrak{l}=\diag(-s,2s,-3s)$,
 $(\alpha,\gamma,\delta,\epsilon)=(\sigma^{12}\otimes a_1+\sigma^{23}\otimes a_2,0,\sigma^1\otimes a_2,0)$,
\item $D_\mathfrak{l}=\diag(3s,-2s,s)$,
 $(\alpha,\gamma,\delta,\epsilon)=(\sigma^{12}\otimes a_1+\sigma^{23}\otimes a_2,0,\sigma^3\otimes a_1,0)$,
\end{itemize}
$\mathfrak{l}=\mathds{R}^3$, $\mathfrak{a}=\mathds{R}^{1,1}$, $D_\mathfrak{a}=\diag(\pm s,\mp s)$ and
\begin{itemize}
\item $D_\mathfrak{l}=\diag(N_{\pm s/2},\mp \frac{3}{2}s)$,
 $(\alpha,\gamma,\delta,\epsilon)=(\sigma^{12}\otimes a_1+\sigma^{23}\otimes a_2,0,0,0)$, 
 \item $D_\mathfrak{l}\in\left\{ \diag(\pm s,e,\mp s-e), \diag(N_{\pm s},\mp 2s), \diag(\pm s, N_{\mp s/2}), \diag(\pm s,D_{\mp s+id}) \right\}$,\\
$(\alpha,\gamma,\delta,\epsilon)=(0,\sigma^{123},\sigma^1\otimes a_1,0)$ 
\end{itemize}
with $a\leq b\leq c$, $a,b,c\neq 0$, $a+b+c=0$, $d,s>0$, $e\notin\{0,s,-s\}$ 
forms a system of representatives of the isomorphism classes of indecomposable non-abelian metric Lie algebras of dimension $8$ with bijective skewsymmetric derivations.
Here $\{a_1,a_2\}$ is a Witt basis for $\mathfrak a=\mathds R^{1,1}$ and an orthonormal basis for $\mathfrak a=\mathds R^2$ or $\mathds R^{2,0}$.
Moreover, denote $\{\sigma^1,\sigma^2,\sigma^3\}$ the dual basis of the standard basis of $\mathds R^3$.
\end{theorem}

\begin{theorem}\label{thm:8}
There are no non-abelian metric symplectic Lie-Algebras of index smaller than three.
\end{theorem}

\begin{theorem}\label{thm:index3}
The only indecomposable non-abelian metric symplectic Lie algebras with index 3 (up to isomorphism) are given as the six-dimensional metric Lie algebra $\mathfrak d_{0,\sigma^{123}}(\mathds R^3,0)$ with derivations $D_{0,0}(D_\mathfrak l,0)$, where $D_\mathfrak l\in\{\diag(a,b,c),\diag(N_b,-2b),\diag(D_{b+is},-2b)\}$, 
or the eight-dimensional metric Lie algebra $\mathfrak d_{\sigma^{13}\otimes a_1+\sigma^{23}\otimes a_2,0}(\mathds R^3,\mathds R^2)$ with derivations
$D_{0,0}(\diag(D_{b+is},-b),D_{is})$
for $a\leq b\leq c$, $a,b,c\neq 0$, $a+b+c=0$ and $s>0$.
Here $\{a_1,a_{2}\}$ is an orthonormal basis of $\mathfrak a=\mathds R^{2}$ and $\{\sigma^1,\sigma^2,\sigma^3\}$ the dual basis of the standard basis of $\mathds R^3$.
\end{theorem}

\begin{proof}[Proof of all theorems]
The computations are so overseeable that we don't have to restrict the computations to indecomposable metric symplectic Lie algebras. We simply compute all non-abelian metric symplectic Lie algebras and can easily decide afterwords which are indecomposable.

For every $4$-dimensional non-abelian nilpotent Lie algebra $\mathfrak{l}$ there is an $m\in\mathds{N}$ such that $\mathfrak{l}^{m+1}$is one-dimensional.
Thus $Z^2_{Q+}(\mathfrak l,D_\mathfrak l,0)_b$ is empty for every derivation $D_\mathfrak{l}$ of $\mathfrak{l}$, since ($A_m$) is not satisfied.
For $\mathfrak l=\mathds R^4$ and $\mathfrak{a}=\{0\}$ we have for every $\gamma\in C^3(\mathfrak{l})$ an isomorphism of pairs $(S,U)$ such that
$(S,U)^*(0,\gamma)=(0,0)$ or $(S,U)^*(0,\gamma)=(0,\sigma^2\wedge\sigma^3\wedge\sigma^4)$.
Since both cocycles $(0,0)$ and $(0,\sigma^2\wedge\sigma^3\wedge\sigma^4)$ do not satisfy ($A_0$), the set $H^2_{Q+}(\mathds R^4,D_\mathfrak l,0)_b$ is empty.

\begin{lemma}\label{lm:40}
The set of balanced cocycles $Z^2_{Q+}(\mathds R,D_\mathfrak l,\mathfrak a)_b$ and $Z^2_{Q+}(\mathds R^2,D_\mathfrak l,\mathfrak a)_b$ does not contain any elements for every bijective derivation $D_\mathfrak l$ of $\mathds R^2$ and for every abelian metric Lie algebra $\mathfrak a$ with bijective skewsymmetric derivation $D_\mathfrak a$.
\end{lemma}

\begin{proof}
Since $H^2_Q(\mathds R,\phi_\mathfrak l,\mathfrak a)$ contains only the element $[0,0]$ for every metric vector space $\mathfrak a$ and therefor does not have any balanced elements, the set $Z^2_{Q+}(\mathds R,D_\mathfrak l,\mathfrak a)_b$ is also empty.
Now, suppose $(\alpha,\gamma,\delta,\epsilon)\in Z^2_{Q+}(\mathds R^2,D_\mathfrak l,\mathfrak a)$ and let ${D_\mathfrak l}_s$ be diagonalizable over $\mathds R$ at the moment. I. e. there is a diagonal basis $\{L_1,L_2\}$ of $\mathds R^2$ for ${D_\mathfrak l}_s$ such that ${D_\mathfrak l}_sL_1=\lambda_1 L_1$ and ${D_\mathfrak l}_sL_2=\lambda_2 L_2$.
Because of $D^\circ_s\alpha=0$ we have
\begin{align*}
0&=\langle (D^\circ_s\alpha)(L_1,L_2),\alpha(L_1,L_2)\rangle\\
&=\langle {D_\mathfrak a}_s(\alpha(L_1,L_2))-\alpha({D_\mathfrak  l}_sL_1,L_2)-\alpha(L_1,{D_\mathfrak l}_sL_2),\alpha(L_1,L_2)\rangle\\
&=\langle -\alpha({D_\mathfrak  l}_sL_1,L_2)-\alpha(L_1,{D_\mathfrak l}_sL_2),\alpha(L_1,L_2)\rangle\\
&=-(\lambda_1+\lambda_2)\langle\alpha(L_1,L_2),\alpha(L_1,L_2)\rangle.
\end{align*}
So, on the one hand
$\alpha=0$ (which is in addition the case for $\lambda_1+\lambda_2=0$, since $D^\circ\alpha=0$ and $D_\mathfrak a$ is bijective) or on the other hand
$\alpha(\mathds R^2,\mathds R^2)$ is an one-dimensional isotropic subspace of $\mathfrak a$.

Now, suppose ${D_\mathfrak l}_s$ is not diagonalizable over $\mathds R$. This means that there is a basis $\{L_1,L_2\}$ of $\mathfrak l$ such that ${D_\mathfrak l}_sL_1=aL_1+bL_2$ and ${D_\mathfrak l}_sL_2=-bL_1+aL_2$ for $a,b\in\mathds R$, $b\neq 0$.
Then analogous
\[0=\langle (D^\circ_s\alpha)(L_1,L_2),\alpha(L_1,L_2)\rangle=-2a\langle\alpha(L_1,L_2),\alpha(L_1,L_2)\rangle.\]
For $\alpha\neq 0$ we obtain that $2a\neq 0$ is an eigenvalue of ${D_\mathfrak a}_s$ because of $D^\circ_s\alpha(L_1,L_2)=0$. Then $\alpha(\mathfrak l,\mathfrak l)$ is an one-dimensional isotropic subspace.
Moreover, the $3$-Form $\gamma$ on $\mathds{R}^2$ is trivial. 
Thus, for every of these cases the condition ($A_0$) or ($B_0$) is not satisfied and hence
$(\alpha,\gamma,\delta,\epsilon)\notin Z^2_{Q+}(\mathds R^2,D_\mathfrak l,\mathfrak a)_b$.
\end{proof}

Since the dimension of $\mathfrak l$ is limited by the index, we obtain directly Theorem \ref{thm:8} from Lemma \ref{lm:40}.
In addition Lemma \ref{lm:40} shows that there is no symplectic metric Lie algebra of dimension less than $6$, except for abelian ones.

Denote $\mathfrak h_3$ the three-dimensional Heisenberg algebra given by  $[e_1,e_2]=e_3$.

\begin{lemma}
The set of cocycles $Z^2_{Q+}(\mathfrak h_3,D_\mathfrak l,\mathfrak a)$ does not contain any balanced elements for every bijective derivation $D_\mathfrak l$ on the Heisenberg algebra $\mathfrak h_3$ and every abelian metric Lie algebra $\mathfrak a$ with skewsymmetric bijective derivation.
\end{lemma}

\begin{proof}
Suppose $(\alpha,\gamma,\delta,\epsilon)\in Z^2_{Q+}(\mathfrak h_3,D_\mathfrak l,\mathfrak a)$.
If $\alpha(\mathfrak h_3,\mathfrak z(\mathfrak h_3))=0$, then condition ($A_1$) is not satisfied.
Thus $(\alpha,\gamma,\delta,\epsilon)\notin Z^2_{Q+}(\mathfrak h_3,D_\mathfrak l,\mathfrak a)_b$.

Now, suppose $\alpha(\mathfrak h_3,\mathfrak z(\mathfrak h_3))\neq 0$ and assume that ${D_\mathfrak l}_s$ is diagonalizable over $\mathds R$.
Then there is a diagonal basis $\{L_1,L_2,L_3\}$ for ${D_\mathfrak l}_s$ satisfying $[L_1,L_2]=L_3$ and $\alpha(L_1,L_3)\neq 0$.
This means that
\[{D_\mathfrak l}_sL_1=\lambda_1L_1,\quad {D_\mathfrak l}_sL_2=\lambda_2L_2,\quad {D_\mathfrak l}_sL_3=(\lambda_1+\lambda_2)L_3,\]
where $\lambda_1,\lambda_2,\lambda_1+\lambda_2\neq 0$.
Because of
\begin{align*}
{D_\mathfrak a}_s\alpha(L_1,L_3)=D^\circ_s\alpha(L_1,L_3)+\alpha({D_\mathfrak l}_sL_1,L_3)+\alpha(L_1,{D_\mathfrak l}_sL_3)=(2\lambda_1+\lambda_2)\alpha(L_1,L_3)
\end{align*}
we obtain that $2\lambda_1+\lambda_2\neq 0$ is an eigenvalue of ${D_\mathfrak a}_s$. Thus $\alpha(L_1,L_3)\neq 0$ is isotropic as an eigenvector for $2\lambda_1+\lambda_2$, since ${D_\mathfrak a}_s$ is skewsymmetric.
Moreover,
\begin{align*}
0&=\langle{D_\mathfrak a}_s\alpha(L_1,L_3),\alpha(L_2,L_3)\rangle+\langle\alpha(L_1,L_3),{D_\mathfrak a}_s\alpha(L_2,L_3)\rangle\\
&=(2\lambda_1+\lambda_2)\langle \alpha(L_1,L_3),\alpha(L_2,L_3)\rangle+(\lambda_1+2\lambda_2)\langle\alpha(L_1,L_3),\alpha(L_2,L_3)\rangle\\
&=3(\lambda_1+\lambda_2)\langle\alpha(L_1,L_3),\alpha(L_2,L_3)\rangle.
\end{align*}
Since $D_\mathfrak l$ is bijective, it holds that $\lambda_1+\lambda_2\neq 0$ and hence $\langle\alpha(L_1,L_3),\alpha(L_2,L_3)\rangle=0$.
Now, assume that ${D_\mathfrak l}_s$ is not diagonalizable over $\mathds R$. Then there is a basis $\{L_1,L_2,L_3\}$ of $\mathfrak l$, which satisfies ${D_\mathfrak l}_sL_1=aL_1+bL_2$, ${D_\mathfrak l}_sL_2=-bL_1+aL_2$, ${D_\mathfrak l}_sL_3=2aL_3$ for $a,b\neq 0$ and $[L_1,L_2]=L_3$, since the center of $\mathfrak h_3$ is invariant under $D_\mathfrak l$.

For $i,j=1,2$ we consider 
\[0=\langle D^\circ_s\alpha(L_i,L_3),\alpha(L_j,L_3)\rangle\]
and obtain analogous that
\[\langle\alpha(L_1,L_3),\alpha(L_1,L_3)\rangle=\langle\alpha(L_1,L_3),\alpha(L_2,L_3)\rangle=0.\]
So, $\alpha(L_1,L_3)$ is orthogonal to every element in $\alpha(\mathfrak l,\mathfrak z(\mathfrak l))$ and ($B_1$) is not satisfied.
Thus $(\alpha,\gamma,\delta,\epsilon)\notin Z^2_{Q+}(\mathfrak l,D_\mathfrak l,\mathfrak a)_b$.
\end{proof}

Since $H^2_{Q+}(\mathfrak l,D_\mathfrak l,0)_b$ is empty for every $4$-dimensional Lie algebra $\mathfrak l$ and the Lie algebras $\mathds{R}$, $\mathds R^2$ and also $\mathfrak h_3$ are not admissable, we obtain Theorem \ref{thm:dim6} by determining $H^2_{Q+}(\mathds R^3,D_\mathfrak l,0)_b/G$.
For Theorem \ref{thm:dim8} it remains to determine $H^2_{Q+}(\mathds R^3,D_\mathfrak l,\mathfrak a)_b/G$ with $\mathfrak a\in\{\mathds{R}^2,\mathds R^{1,1},\mathds R^{2,0}\}$
and Theorem \ref{thm:index3} follows by calculating the $G$-orbits of $H^2_{Q+}(\mathds R^3,D_\mathfrak l,\mathds R^{2n})_b$ for $n\in\mathds{N}$. 

For $(\alpha,\gamma,\delta,\epsilon)\in Z^2_{Q+}(\mathds{R}^3,D_\mathfrak{l},\mathfrak{a})_b$  we have
\begin{align}\label{eq:a1}
0=&-\langle(D^\circ_s\alpha)(L_1,L_2),\alpha(L_3,L_4)\rangle-\langle\alpha(L_1,L_2),(D^\circ_s\alpha)(L_3,L_4)\rangle\\
=&\langle\alpha({D_\mathfrak{l}}_sL_1,L_2)+\alpha(L_1,{D_\mathfrak{l}}_sL_2),\alpha(L_3,L_4)\rangle
+\langle\alpha(L_1,L_2),\alpha({D_\mathfrak{l}}_sL_3,L_4)+\alpha(L_3,{D_\mathfrak{l}}_sL_4)\rangle\nonumber
\end{align}
for vectors $L_1,\dots,L_4\in\mathfrak l$ because of $D^\circ_s\alpha=0$ and the skewsymmetry of ${D_\mathfrak a}_s$.
Moreover, we have
\begin{align}\label{eq:a2}
0=-\langle(D^\circ_s\alpha)(L_1,L_2),\alpha(L_1,L_2)\rangle
=\langle\alpha({D_\mathfrak{l}}_sL_1,L_2)+\alpha(L_1,{D_\mathfrak{l}}_sL_2),\alpha(L_1,L_2)\rangle.
\end{align}

\begin{lemma}
For $(\alpha,\gamma,\delta,\epsilon)\in Z^2_{Q+}(\mathds{R}^3,D_\mathfrak{l},\mathfrak{a})_b$ the conditions
 \begin{itemize}
  \item $\alpha=0$,
  \item $\gamma\neq 0$,
  \item $\tr(D_\mathfrak{l})=0$
 \end{itemize}
are equivalent.
\end{lemma}

\begin{proof}
Suppose $(\alpha,\gamma,\delta,\epsilon)\in Z^2_{Q+}(\mathds{R}^3,D_\mathfrak{l},\mathfrak{a})_b$.
Let $D_\mathfrak{l}$ be a bijective zero-trace matrix.
At first, let $D_\mathfrak{l}$ be diagonalizable over $\mathds R$. 
Denote $\{L_1,L_2,L_3\}$ the basis of eigenvectors of ${D_\mathfrak{l}}_s$ for the corresponding eigenvalues $\lambda_1,\lambda_2,\lambda_3\neq 0$.
Consider
\begin{align*}
0=&\langle\alpha({D_\mathfrak{l}}_sL_i,L_j)+\alpha(L_i,{D_\mathfrak{l}}_sL_j),\alpha(L_p,L_k)\rangle
+\langle\alpha(L_i,L_j),\alpha({D_\mathfrak{l}}_sL_p,L_k)+\alpha(L_p,{D_\mathfrak{l}}_sL_k)\rangle\\
=&(\lambda_i+\lambda_j+\lambda_p+\lambda_k)\langle\alpha(L_i,L_j),\alpha(L_p,L_k)\rangle
\end{align*}
(compare to equation (\ref{eq:a1})).
Because of $\lambda_1+\lambda_2+\lambda_3=0$ we have $\langle\alpha(\mathfrak{l},\mathfrak{l}),\alpha(\mathfrak{l},\mathfrak{l})\rangle=0$.
Now ($B_0$) implies that $\alpha(\mathfrak{l},\mathfrak{l})$ is non-degenerate with respect to $\langle\cdot,\cdot\rangle_\mathfrak{a}$ and thus $\alpha=0$.
If the zero-trace matrix ${D_\mathfrak l}_s$ is not diagonalizable over $\mathds R$, then there is a basis $\{L_1,L_2,L_3\}$ of $\mathfrak l$ such that ${D_\mathfrak l}_sL_1=aL_1+bL_2$, ${D_\mathfrak l}_sL_2=-bL_1+aL_2$ and ${D_\mathfrak l}_sL_3=-2aL_3$ for $a,b\neq 0$.
As above, we can consider the equation
\[0=\langle \alpha({D_{\mathfrak l}}_sL_i,L_j)+\alpha(L_i,{D_{\mathfrak l}}_sL_j),\alpha(L_p,L_k)\rangle+\langle\alpha(L_i,L_j),\alpha({D_{\mathfrak l}}_sL_p,L_k)+\alpha(L_p,{D_{\mathfrak l}}_sL_k)\rangle\]
for $i,j,k,l\in\{1,2,3\}$ and obtain
$\langle\alpha(\mathfrak l,\mathfrak l),\alpha(\mathfrak l,\mathfrak l)\rangle=0$.
Thus $\alpha=0$ because of ($B_0$).

Now, suppose $\alpha=0$. Using ($A_0$) we obtain $\gamma\neq 0$.
Finally from $\gamma\neq 0$ follows directly that $D_\mathfrak{l}$ has zero trace by using
\[0=D^\circ_s\gamma=-\tr(D_\mathfrak{l})\gamma.\]
\end{proof}

Suppose $(\alpha,\gamma,\delta,\epsilon)\in Z^2_{Q+}(\mathds{R}^3,D_\mathfrak{l},\mathfrak{a})_b$.
If ${D_\mathfrak l}_s$ is not diagonalizable over $\mathds R$, then we consider the complexification of the derivations and differential forms.
Denote $\{L_1,L_2,L_3\}$ a basis of eigenvectors of ${D_\mathfrak{l}}_s$ for the corresponding eigenvalues $\lambda_1,\lambda_2,\lambda_3\neq 0$.
We consider
\begin{align}\label{eq:epsilon0}
0=(D^\circ_s\epsilon)(L_i,L_j)=-\epsilon({D_\mathfrak{l}}_sL_i,L_j)-\epsilon(L_i,{D_\mathfrak{l}}_sL_j)=-(\lambda_i+\lambda_j)\epsilon(L_i,L_j).
\end{align}

Of course, $\alpha=0$ and $\delta=0$ for $\mathfrak{a}=0$. So $D_\mathfrak{l}$ is a zero-trace matrix and $\gamma=k\sigma^{123}\neq 0$, where $\{\sigma^1,\sigma^2,\sigma^3\}$ is the dual of the standard basis of $\mathds{R}^3$.
Since the sum of two eigenvalues of ${D_\mathfrak{l}}_s$ is not zero, we have $\epsilon=0$.
Thus $(\alpha,\gamma,\delta,\epsilon)=(0,k\sigma^{123},0,0)$ with $k\neq 0$.
Finally
$(k^{-\frac{1}{3}} \id,\id)\in G(\mathds R^3,D_\mathfrak l,0)$
and we get 
\begin{align}\label{eq:e1}
(k^{-\frac{1}{3}} \id,\id)^*(0,k\sigma^{123},0,0)=(0,\sigma^{123},0,0).
\end{align}

Now, assume $\mathfrak{a}\neq\{0\}$.
At first, let us assume that $(\alpha,\gamma,\delta,\epsilon)\in Z^2_{Q+}(\mathds{R}^3,D_\mathfrak{l},\mathfrak{a})_b$ for a bijective zero- trace-matrix $D_\mathfrak{l}$.
Then, again, $\alpha=0$, $\gamma=k\sigma^1\wedge\sigma^2\wedge\sigma^3\neq 0$ and $\epsilon=0$ because of equation (\ref{eq:epsilon0}).
Now, consider 
\begin{align}\label{eq:deltaneq0}
0=(D^\circ_s\delta)(L_i)={D_\mathfrak{a}}_s(\delta(L_i))-\delta({D_\mathfrak{l}}_sL_i)={D_\mathfrak{a}}_s\delta(L_i)-\lambda_i\delta(L_i).
\end{align}
The eigenvalues of $D_\mathfrak{a}$ are purely imaginary for $\mathfrak{a}\in\{\mathds{R}^2,\mathds{R}^{2,0}\}$. But the eigenvalues of the zero-trace matrix ${D_\mathfrak{l}}_s$ can't be purely imaginary, so we have $\delta=0$.
Because of $(k^{-\frac{1}{3}} \id,\id)\in G(\mathds R^3,D_\mathfrak l,0)$ and equation (\ref{eq:e1}) the cohomology class \[H^2_{Q+}(\mathds{R}^3,D_\mathfrak{l},\mathfrak{a})_b/G(\mathds{R}^3,D_\mathfrak{l},\mathfrak{a})\] consists of only one element for bijective zero-trace matrices $D_\mathfrak{l}$ and $\mathfrak{a}\in\{\mathds{R}^2,\mathds{R}^{2,0}\}$, which is represented by $[0,\sigma^{123},0,0]$. The same holds for an arbitrary euclidian vector space $\mathfrak a$ of even dimension.

Now, assume $\mathfrak{a}=\mathds{R}^{1,1}$.
If there is no real number, which is an eigenvalue of ${D_\mathfrak a}_s=D_\mathfrak a$ and ${D_\mathfrak l}_s$ at the same time, then $\delta=0$ because of equation (\ref{eq:deltaneq0}).
If there is such a real number, then this number is unique, since both eigenvalues $\pm s\neq 0$ of $D_\mathfrak a$ can't be eigenvalues of the bijective zero-trace matrix ${D_\mathfrak l}_s$ at the same time. Thus $\delta(\mathfrak l)$ is either $\{0\}$ or it spanns a one-dimensional eigenspace of $D_\mathfrak a$.
Assume $\delta\neq 0$. If there is no eigenvector $\hat v$ of $D_\mathfrak l$ such that $\delta(\hat v)\neq 0$, then there is a vector $v$ in the generalized eigenspace $\pm s$ of $D_\mathfrak l$ satisfying $\delta(v)\neq 0$ because of equation (\ref{eq:deltaneq0}). Especially the generalized eigenspace for $\pm s$ is two-dimensional and the generalized eigenspace for $\mp 2s$ is one-dimensional.
The vector $v$ is no eigenvector of $D_\mathfrak l$. Thus there is an eigenvector $\tilde v$ of $D_\mathfrak l$ with $D_\mathfrak lv=\pm s+\tilde v$.
We define $\tau\in C^1(\mathds{R}^3,\phi_\mathfrak l,\mathds{R}^{1,1})$ by $\tau(\tilde v)=\delta(v)$, $\tau(v)=0$ and $\tau(w)=0$ for an eigenvector $w$ for the eigenvalue $\mp2s$.
Then $D^\circ_s\tau=0$ because of $D^\circ_s\delta=0$ and 
\[D^\circ\tau(v)=D^\circ_s\tau(v)-\tau(\tilde v)=-\delta(v).\]
Hence $(0,\gamma,\delta,0)(\tau,0)=(0,\gamma,0,0)$ and thus $(0,\gamma,\delta,0)$ and $(0,\gamma,0,0)$ are equivalent.
Using the morphism of pairs $(S,U)=(k^{-\frac{1}{3}} \id,\id)$ we obtain
\[(S,U)^*(0,\gamma,0,0)=(0,\sigma^{123},0,0).\]

Now, let there be a real eigenvector $D_\mathfrak{l}$, which spanns the one-dimensional subspace $\delta(\mathfrak{l})$.
This case can't be reduced to the previous case, since this property is invariant under the $G$-action. Moreover, it is invariant under equivalence, since
\[(\delta+D^\circ\tau)(v)=\delta(v)+D_\mathfrak a\tau(v)-\tau(D_\mathfrak l(v))=\delta(v)+{D_\mathfrak a}_s\tau(v)-\tau({D_\mathfrak l}_s(v))=(\delta+D^\circ_s\tau)(v)=\delta(v)\]
holds for  an eigenvector $v$ of $D_\mathfrak{l}$.
We choose an eigenvector $v_1\in\mathfrak l$ of $D_\mathfrak l$, which satisfies $\delta(v_1)\neq 0$.
Then we extend $v_1$ to a real Jordan or orthonormal basis $\{v_1,v_2,v_3\}$ of $D_\mathfrak{l}$ satisfying $\delta(v_2)=\delta(v_3)=0$.
We obtain $\gamma=k\sigma^{123}\neq 0$.
Furthermore, we choose a vector $a_2$ in $\mathfrak a$ such that $a_1:=\delta(v_1)$ and $a_2$ form a Witt basis of $\mathfrak{a}=\mathds{R}^{1,1}$.
We set 
\[(S,U)=(k^{-\frac{1}{3}} \id,\diag(k^{\frac{1}{3}},k^{-\frac{1}{3}}))\in G(\mathds R^3,D_\mathfrak l,\mathds R^{1,1})\]
and obtain \[(S,U)^*(0,\gamma,\delta,0)=(0,\sigma^{123},\delta,0).\]
So this shows that every $(0,\gamma,\delta,0)$ with an one-dimensional subspace $\delta(\mathfrak{l})$, which is spanned by an eigenvector of $D_\mathfrak{l}$, lives in the same $G$-orbit.
Finally we have shown that for a bijectice zero-trace matrix $D_\mathfrak{l}$
\[H^2_{Q+}(\mathds R^3,D_\mathfrak l,\mathds R^{1,1})_b/G(\mathds R^3,D_\mathfrak l,\mathds R^{1,1})\] consists of only one element, if no eigenvalue of ${D_\mathfrak l}_s$ is one of $D_\mathfrak a$.
If there is an eigenvalue of ${D_\mathfrak l}_s$, which is one of $D_\mathfrak{a}$, then the cohomology class has two elements.
\\\quad\\
Assume $(\alpha,\gamma,\delta,\epsilon)\in Z^2_{Q+}(\mathds{R}^3,D_\mathfrak{l},\mathfrak{a})_b$ for bijective $D_\mathfrak{l}$ with $\tr(D_\mathfrak{l})\neq 0$.
Then $\gamma=0$ and thus $\alpha\neq 0$.

At first, we assume $\mathfrak{a}=\mathds{R}^2,\mathds{R}^{2,0}$.
Using $D^\circ_s\alpha=0$ we see that $D_\mathfrak{l}$ has the eigenvalues $a+is$, $a-is$ and $-a$ for an $a\neq 0$, where $\pm is$ are the eigenvalues of $D_\mathfrak{a}$ with $s>0$.
Since $D_\mathfrak{l}$ has no purely imaginary eigenvalues and, furthermore, the sum of two eigenvalues is not equal to zero, we get $\delta=0$ and $\epsilon=0$ (also compare to equation (\ref{eq:epsilon0}) and (\ref{eq:deltaneq0})).

Now, we consider 
$H^2_{Q+}(\mathds R^3,D_\mathfrak l,\mathfrak a)_b/G$, where
\[D_\mathfrak l=\begin{pmatrix}a&-s&\\s&a&\\&&-a\end{pmatrix},\quad D_\mathfrak a=\begin{pmatrix}&-s\\s&\end{pmatrix}\]
with $s>0$ and $a\neq 0$.
We have $\alpha(e_1,e_2)=0$, since $D^\circ_s\alpha=0$.
It also holds that $\alpha(e_1,e_3)\neq 0$ and $\alpha(e_2,e_3)\neq 0$ are linearly independent because of ($A_0$).
Moreover, $\alpha(e_1,e_3)$ and $\alpha(e_2,e_3)$ are orthogonal to each other, since
\[0=\langle \alpha({D_\mathfrak l}_se_1,e_3)+\alpha(e_1,{D_\mathfrak l}_se_3),\alpha(e_1,e_3)\rangle=-s\langle\alpha(e_2,e_3),\alpha(e_1,e_3)\rangle,\]
and they have the same length because of
\begin{align*}
0&=\langle \alpha({D_\mathfrak l}_se_1,e_3)+\alpha(e_1,{D_\mathfrak l}_se_3),\alpha(e_2,e_3)\rangle+0=\langle \alpha(e_1,e_3),\alpha({D_\mathfrak l}_se_2,e_3)+\alpha(e_2,{D_\mathfrak l}_se_3)\rangle\\
&=s\langle \alpha(e_2,e_3),\alpha(e_2,e_3)\rangle-s\langle \alpha(e_1,e_3),\alpha(e_1,e_3)\rangle.
\end{align*}
We normalize the vectors $\{\alpha(e_1,e_3),\alpha(e_2,e_3)\}$ to an orthonormal basis $\{v_1,v_2\}$ of $\mathfrak a$.
Then $\alpha=\mu(\sigma^{13}\otimes v_1 +\sigma^{23}\otimes v_2)$ with $\mu\neq 0$ because of $D^\circ_s\alpha=0$
and using 
\[(S,U)=(\diag(1,1,\mu^{-1}),\id)\in G\]
yields
\[(S,U)^*(\alpha,0,0,0)=(\sigma^{13}\otimes v_1+\sigma^{23}\otimes v_2,0,0,0).\]
Hence, we can represent $H^2_{Q+}(\mathds{R}^3,D_\mathfrak l,\mathfrak{a})_b/G$ by $[\sigma^{13}\otimes a_1+\sigma^{23}\otimes a_2,0,0,0]$.
Here $\{a_1,a_2\}$ denotes the standard basis of $\mathfrak a$.
For an arbitrary euclidian vectorspace $\mathfrak a$ of even dimension with orthonormal basis $\{a_1,\dots,a_{2n}\}$ we have analogous 
\[H^2_{Q+}(\mathds{R}^3,D_\mathfrak l,\mathfrak{a})_b/G=\{[\sigma^{13}\otimes a_1+\sigma^{23}\otimes a_2,0,0,0]\}.\]

Now, assume $\mathfrak{a}=\mathds{R}^{1,1}$.
Because of $D^\circ_s\alpha=0$ and ($B_0$) we know that the sum of two eigenvalues of ${D_\mathfrak{l}}_s$ equals $s$ and another sum equals $-s$, where $\pm s$ denote the eigenvalues of $D_\mathfrak{a}$ with $s>0$.
This especially implies that the eigenvalues of ${D_\mathfrak{l}}_s$ are real numbers.
Since the sum of two eigenvalues is not zero, we obtain $\epsilon\equiv 0$.

Let $D_\mathfrak l$ be semisimple.
Since the eigenvalues of $D_\mathfrak{l}$ are real numbers, there is a basis of eigenvectors $\{L_1,L_2,L_3\}$.
This basis is w. l. o. g. given such that, the sum of the eigenvalues of  $L_1$ and $L_2$ equals $s$ and the sum of the eigenvalues of $L_2$ and $L_3$ equals $-s$.
Since $\tr D_\mathfrak{l}\neq 0$, we have $\gamma=0$ and $\alpha(\mathfrak{l},\mathfrak{l})\neq 0$ is non-degenerate.
Thus $\alpha(L_1,L_2)$ is an eigenvector of $D_\mathfrak{a}$ for the eigenvalue $s$ and $\alpha(L_2,L_3)$ an eigenvector of $D_\mathfrak{a}$ for $-s$.
Moreover, $\tr D_\mathfrak{l}\neq 0$ implies that the sum of the eigenvalues of $L_1$ and $L_3$ is not equal to $\pm s$ and hence $\alpha(L_1,L_3)=0$.
W. l. o. g. we choose $\{L_1,L_2,L_3\}$ such that $\{\alpha(L_1,L_2),\alpha(L_2,L_3)\}$ is a Witt basis of $\mathfrak{a}$.

If no eigenvalue of $D_\mathfrak l$ is an eigenvalue of $D_\mathfrak a$, then $\delta=0$, because of (\ref{eq:deltaneq0}). If there is such an eigenvalue, then we have on the one hand $\delta=k\sigma^1\otimes a_2$ for $-s$ being an eigenvalue of $D_\mathfrak l$ or on the other hand $\delta=k\sigma^3\otimes a_1$for $s$ being an eigenvalue of $D_\mathfrak l$, since $D^\circ\delta=0$. Here $k\in\mathds R$.
For $\delta\neq 0$ the cocycles $(\alpha,0,0,0)$ and $(\alpha,0,\delta,0)$ are not equivalent, nor they lie in the same $G$-orbit. Moreover,  $(S,U)^*(\alpha,0,\delta,0)=(\alpha,0,k^{-1}\delta,0)$ for $(S,U)=(\diag(k^{-1},k,k^{-1}),\operatorname{id})\in G$.
For a semisimple $D_\mathfrak l$, whose trace doesn't vanish, the cohomology class $H^2_{Q+}(\mathfrak l,D_\mathfrak l,\mathfrak a)_b/G$ consists of
\begin{itemize}
 \item no elements, if one eigenvalue of $D_\mathfrak a$ is not the sum of two eigenvalues of $D_\mathfrak l$,
 \item exactly one element, if both eigenvalue of $D_\mathfrak a$ are the sum of two eigenvalues of $D_\mathfrak l$, but there is no eigenvalue of $D_\mathfrak l$ which is also an eigenvalue of $D_\mathfrak a$ and
 \item exactly two elements, otherwise.
\end{itemize}
For a $D_\mathfrak l$, which is not semisimple, we have 
\[0=(d\delta-D^\circ\alpha)(L_2,L_3)=D^\circ_s\alpha(L_2,L_3)-\alpha(L_1,L_3)=-\alpha(L_1,L_3).\]
Here $\{L_1,L_2,L_3\}$ denotes a Jordan basis of $D_\mathfrak l$, where $D_\mathfrak lL_1=\pm \frac{s}{2}L_1$, $D_\mathfrak lL_2=\pm \frac{s}{2}L_2+L_1$ and $D_\mathfrak lL_3=\mp \frac{3}{2}sL_3$.
It is easy to see that there is a Jordan basis such that $\alpha(L_1,L_2)$ is an eigenvector for $\pm s$, $\alpha(L_2,L_3)$ is one for $\mp s$ and $\{\alpha(L_1,L_2),\alpha(L_2,L_3)\}$ is a Witt basis of $\mathds R^{1,1}$. Since $\pm s$ is no eigenvalue of $D_\mathfrak l$, we obtain $\delta=0$ using equation (\ref{eq:deltaneq0}).
Thus, in this case $H^2_{Q+}(\mathds R^3,D_\mathfrak l,\mathds R^{1,1})_b/G$ consists of exactly one element.
\\\quad\\
At this moment we determined the $G$-orbits of $H^2_{Q+}(\mathds R^3,D_\mathfrak l,\mathfrak a)_b$ for every derivation $D_\mathfrak l$ of $\mathfrak l=\mathds R^3$ and every trivial $(\mathfrak l,D_\mathfrak l)$-moduls $(\mathfrak a,D_\mathfrak a)$, where $\mathfrak a\in\{0,\mathds R^2,\mathds R^{1,1},\mathds R^{2,0},\mathds R^{2n},\}$ and $D_\mathfrak a$ is bijective.
Choosing representatives of the conjugation classes of the derivations of  $\mathfrak l$ and $\mathfrak a$ yield the theorems.
\end{proof}



\begin{thebibliography}{99}

\bibitem{BBM07}
Bajo, I., Benayadi, S., Medina, A.,
\emph{Symplectic structures on quadratic Lie algebras},
Journal of Algebra \textbf{316}(1) (2007), 174-188


\bibitem{BC13}
Baues, O., Cortès, V.,
\emph{Symplectic Lie groups},
Astérisque \textbf{379} (2016) 96 pages


\bibitem{Bo97}
Bordemann, M.,
\emph{Non-degenerate invariant bilinear forms on nonassociative algebras},
Acta Math. Univ. Com. \textbf{LXVI}, 2, (1997), 151-201.

\bibitem{Bu06}
Burde, D.,
\emph{Characteristically nilpotent Lie algebras and symplectic structures},
Forum Math. \textbf{18} (2006), 769-787

\bibitem{CP80}
Cahen M., Parker M.,
\emph{Pseudo-Riemannian symmetric spaces},
Memoirs of the American Mathematical Society \textbf{24} (1980) no. 229.

\bibitem{CS09}
Campoamor-Stursberg, R.,
\emph{Symplectic forms on six-dimensional real solvable Lie algebras I},
Algebra Colloq. \textbf{16} (2009), no. 2, 253-266.

%

\bibitem{DM962}
Dardié, J. M., Medina, A.,
\emph{Algèbres de Lie kählériennes et double extension},
Journal of Algebra \textbf{185} (1996), 774-795.

\bibitem{DM96}
Dardié, J. M., Medina, A.,
\emph{Double extension symplectique dun groupe de Lie symplectique},
Advances in Mathematics \textbf{117} (1996), no. 2, 208-227.

\bibitem{GJK01}
Gómez, J. R., Jiménez-Merchán, A., Khakimdjanov, Y.,
\emph{Symplectic structures on the filiform Lie algebras},
Journal of Pure and Applied Algebra \textbf{156} (2001), no. 1, 15-31.

\bibitem{GB87}
Goze, M., Bouyakoub,  A.,
\emph{Sur les algèbres de Lie munies d'une forme symplectique}
Rend. Sem. Fac. Sc. Univ. Cagliari \textbf{57} (1987), no. 1, 85-97.


\bibitem{Ja55}
Jacobson, N.,
\emph{A note on automorphisms and derivations of Lie algebras},
Proceedings of the American Mathematical Society \textbf{6} (1955), no. 2, 281-283.

\bibitem{Ka07}
Kath, I.,
\emph{Nilpotent metric Lie algebras of small dimension},
Journal of Lie Theory \textbf{17} (2007), no. 1, 41-61.

\bibitem{KO06}
Kath, I., Olbrich, M.,
\emph{Metric Lie algebras and quadratic extensions},
Transf. Groups \textbf{11}, (2006), no. 1, 87-131.

\bibitem{KO08}
Kath, I., Olbrich M.,
\emph{The classification problem for pseudo-Riemannian symmetric spaces},
Recent developments in pseudo-Riemannian geometry, ESI Lectures in Mathematics and Physics (EMS Publishing House, Zürich, 2008), 1-52.

\bibitem{GKM04}
Khakimdjanov, Y., Goze M., Medina, A.,
\emph{Symplectic or Contact Structures on Lie Groups},
Differential Geometry and its Applications \textbf{21} (2004), no. 1, 41-54.

\bibitem{MR83}
Medina, A., Revoy, P.,
\emph{Caractérisation des groupes de Lie ayant une pseudométrique bi-invariante},
Applications Séminaire Sud-Rhodanian de Géometrie III ('Géometrie symplectique et du contact autour du theorème de Poincaré-Birkhoff'), Lyon, 12.-17.10.1983, Paris, 1984.

\bibitem{MR85}
Medina, A., Revoy, P.,
\emph{Algèbres de Lie et produit scalaire invariant},
Ann. Sci. Ecole Normale Supp. (4) \textbf{18}(3) (1985), 553-561.

\bibitem{MR91}
Medina, A., Revoy, P.,
\emph{Groupes de Lie à structure symplectique invariante}, Symplectic geometry, groupoids, and integrable systems (Berkeley, CA, 1989), 247-266, Mathematical Sciences Research Institute Publications,
\textbf{20}, Springer, New York, 1991

%
%
\bibitem{Ov06}
Ovando, G.,
\emph{Four dimensional symplectic Lie algebras},
Beiträge zur Algebra und Geometrie \textbf{47} (2006), no. 2, 419-434.

\bibitem{Sa01}
Salamon, S.,
\emph{Complex structures on nilpotent Lie algebras},
Journal of Pure and Applied Algebra \textbf{157} (2001), no. 2-3, 311-333.

%
%
\end{thebibliography}
\end{document}